\documentclass[a4paper, reqno, 10pt]{amsart}
\usepackage{amsthm}
\usepackage{enumerate}
\usepackage{amsmath}
\usepackage{verbatim}
\usepackage{times,latexsym,amssymb, wasysym}

\usepackage[colorlinks,pdfpagelabels,pdfstartview=FitH,bookmarksopen=true,bookmarksnumbered=true,linkcolor=blue,plainpages=false,
hypertexnames=false,citecolor=red]{hyperref}

\numberwithin{equation}{section}
\usepackage[usenames,dvipsnames]{color}

\newtheorem{Theorem}{Theorem}[section]
\newtheorem{Proposition}[Theorem]{Proposition}

\newtheorem{Lemma}[Theorem]{Lemma}
\newtheorem{Corollary}[Theorem]{Corollary}

\theoremstyle{definition}
\newtheorem{Definition}[Theorem]{Definition}
\newtheorem{Construction}[Theorem]{Construction}

\theoremstyle{remark}
\newtheorem{Remark}[Theorem]{Remark}

\def\dx{\,dx}

\def\dt{\,dt}
\def\ds{\,ds}

\DeclareMathOperator{\divergence}{div}
\DeclareMathOperator{\diam}{diam}
\DeclareMathOperator*{\osc}{osc}
\newcommand{\essosc}{\operatornamewithlimits{ess\, osc}}

\def\mint_#1{\mathchoice%
          {\mathop{\kern 0.2em\vrule width 0.6em height 0.69678ex depth -0.58065ex
                  \kern -0.8em \intop}\nolimits_{\kern -0.4em#1}}%
          {\mathop{\kern 0.1em\vrule width 0.5em height 0.69678ex depth -0.60387ex
                  \kern -0.6em \intop}\nolimits_{#1}}%
          {\mathop{\kern 0.1em\vrule width 0.5em height 0.69678ex
              depth -0.60387ex
                  \kern -0.6em \intop}\nolimits_{#1}}%
          {\mathop{\kern 0.1em\vrule width 0.5em height 0.69678ex depth -0.60387ex
                  \kern -0.6em \intop}\nolimits_{#1}}}

\newcommand{\aveint}[2]{\mathchoice%
          {\mathop{\kern 0.2em\vrule width 0.6em height 0.69678ex depth -0.58065ex
                  \kern -0.8em \intop}\nolimits_{\kern -0.45em#1}^{#2}}%
          {\mathop{\kern 0.1em\vrule width 0.5em height 0.69678ex depth -0.60387ex
                  \kern -0.6em \intop}\nolimits_{#1}^{#2}}%
          {\mathop{\kern 0.1em\vrule width 0.5em height 0.69678ex depth -0.60387ex
                  \kern -0.6em \intop}\nolimits_{#1}^{#2}}%
          {\mathop{\kern 0.1em\vrule width 0.5em height 0.69678ex depth -0.60387ex
                  \kern -0.6em \intop}\nolimits_{#1}^{#2}}}
                 
\let\originalleft\left
\let\originalright\right
\renewcommand{\left}{\mathopen{}\mathclose\bgroup\originalleft}
\renewcommand{\right}{\aftergroup\egroup\originalright}

\newcommand{\loc}{\textnormal{loc}}

\newcommand{\supp}{\textnormal{supp}\,}
\newcommand{\esssup}{\operatornamewithlimits{ess\, sup}}
\newcommand{\essinf}{\operatornamewithlimits{ess\, inf}}

\newcommand{\essliminf}{\operatornamewithlimits{ess\,lim\, inf}}

\newcommand{\R}{{\mathbb R}}
\newcommand{\Q}{{\mathbb Q}}

\newcommand{\N}{{\mathbb N}}

\newcommand{\A}{{\mathcal A}}

\newcommand{\eps}{\varepsilon}

\newcommand{\vs}{\vspace{3mm}}

\begin{document}

\title[]{Obstacle problem\\ for a class of parabolic equations\\ of generalized $p$-Laplacian type}

\author{Casimir Lindfors}
\address{Aalto University, Institute of Mathematics, P.O. Box 11100, FI-00076 Aalto, Finland.}
\email{casimir.lindfors@aalto.fi}

\allowdisplaybreaks
\date{\today}

\begin{abstract}
We study nonlinear parabolic PDEs with Orlicz-type growth conditions. The main result gives the existence of a unique solution to the obstacle problem related to these equations. To achieve this we show the boundedness of weak solutions and that a uniformly bounded sequence of weak supersolutions converges to a weak supersolution. Moreover, we prove that if the obstacle is continuous, so is the solution.
\end{abstract}

\keywords{Degenerate/singular parabolic equations; General growth conditions; Obstacle problem}

\maketitle

\section{Introduction}

In this paper we prove the existence of a unique solution to the obstacle problem related to a wide class of nonlinear parabolic equations with a merely bounded obstacle. If the obstacle is also continuous, we show that the solution inherits the same property. More specifically, we consider equations of the type 
\begin{equation}\label{eq}
\partial_tu-\divergence \A(Du) = 0,
\end{equation}
where $\A$ is a $C^1$ vector field with $\A(\xi)\approx \frac{g(|\xi|)}{|\xi|}\xi$. Here $g\in C^1(\R_+)$ is a positive function satisfying the Orlicz-type growth condition 
\begin{equation}\label{O.property}
 g_0-1\leq\frac{sg'(s)}{g(s)}\leq g_1-1,\qquad s>0
\end{equation}
with $2n/(n+2)<g_0\leq g_1<\infty$. A function $u$ solves the corresponding obstacle problem with the obstacle $\psi$ if it is the smallest $\essliminf$-regularized (see \eqref{essliminfreg}) weak supersolution to \eqref{eq} such that $u\geq\psi$ almost everywhere in $\Omega_T$. 

Equation \eqref{eq} is a generalization of the widely studied evolutionary $p$-Laplace type equations. Indeed, when $g_0=g_1=p$ we have $g(s)=s^{p-1}$ up to a constant. For these equations the existence of a continuous solution to the obstacle problem with a continuous obstacle was proved in~\cite{KKS}. More irregular obstacles are treated in~\cite{LP}. Our proofs are analogous to those in~\cite{KKS,LP}, in fact, the main new ingredients are Theorem~\ref{superlimit}, which tells that a sequence of uniformly bounded weak supersolutions converges pointwise to a weak supersolution, and Theorem~\ref{boundednesstheorem}, stating that a nonnegative weak subsolution is locally bounded. For $p$-Laplace type equations the former was proved in \cite{KKP}, see also~\cite{LM}, for the latter we refer to~\cite{Porzio, DiBenedetto}.

In the study of the evolutionary $p$-Laplacian there is a strong distinction between the degenerate ($p\geq 2$) and singular ($1<p<2$) cases. For the more general equations we are interested in, the main difficulty compared to the $p$-Laplace case arises from the fact that the equation can be both degenerate and singular. Indeed, this is possible when $g_0<2<g_1$, see \cite{BL} for a concrete example. This difficulty can be seen for example in the proof of Theorem~\ref{boundednesstheorem}, where we merely obtain a qualitative bound for subsolutions, contrary to the $p$-Laplace case, see~\cite{DiBenedetto}. Another indication of how problematic the more general growth conditions can be is the fact that the H\"older continuity of solutions to parabolic equations with only measurable coefficients is still an open problem. Purely degenerate and purely singular cases have been treated in~\cite{HL1} and~\cite{HL2}, respectively, see also~\cite{Hwang}. Operators satisfying the more general growth conditions were first systematically studied in~\cite{Lieberman1}. Further developments have been made in~\cite{Cianchi-Mazya, DE, Baroni} in the elliptic setting and, in addition to the ones mentioned above, the parabolic case has been studied in~\cite{LiebPara1, LiebRussia, BKa}. The variational counterpart has been treated in~\cite{DieningBianca, FS, Cianchi, Cianchi-Fusco}. 

Motivation to study equations with more general growth comes from many physical phenomena that cannot be modeled sufficiently accurately using polynomial growth. For example, the stationary, irrotational flow of a compressible fluid can be modeled using an equation 
of the type
\[
\divergence\big[\rho(|Du|^2)Du\big]=0,
\]
where $Du$ is the velocity field of the flow and $|Du|=:q$ the speed of the flow. In this context one introduces the {\em Mach number}
\[
M^2\equiv [M(q)]^2:=-\frac{2 q^2}{\rho(q^2)}\rho'(q^2)
\]
(note that we must have $\rho'<0$). The general theory asserts that a point is {\em elliptic} if $M<1$ and in this case the flow is {\em subsonic}, while if $M>1$ the point is {\em hyperbolic} and the flow there is {\em supersonic}. If $M=1$ the flow is called {\em sonic}. In our context, where $g(s)=\rho(s^2)s$, we compute the Orlicz ratio $sg'(s)/g(s)=1-M(s)^2$. Thus, if we know that the flow maintains a controlled, small speed $q$, then the problem falls in the class of operators we consider. For further details see for instance \cite{Bers, Dong, Acta}.

Obstacle problems are a widely studied topic in the theory of partial differential equations. This is due to the fact that obstacle 
problems have numerous applications in several different areas of science, including physics, chemistry, biology, and even finance. Moreover, obstacle problems have turned out to be a fundamental tool in potential theory. The regularity of solutions to obstacle problems and the related free boundary problems is a classical topic in partial differential equations; for this we refer to~\cite{ACS1, Caffarelli}. For more recent advances in the parabolic setting see~\cite{KKS, LP, Lindqvist2, KMN}. The elliptic case has been treated comprehensively in~\cite{HKM}. 

The paper is organized as follows. In Section~\ref{preliminaries} we introduce the basic properties of the function $g$ and some useful results for the related Orlicz spaces. In Section~\ref{usefulresults} we state known results for weak solutions to \eqref{eq} and prove the convergence result Theorem~\ref{superlimit}. The boundedness of solutions is established in Section~\ref{qualitativeproperties}, and for this we prove an {\em a priori} result (Lemma~\ref{k_lemma}) which we find interesting in its own right. Moreover, we show that weak supersolutions to \eqref{eq} are lower semicontinuous (Theorem~\ref{superlsc}). Finally, in Section~\ref{obstacleproblem} we prove the existence result for the obstacle problem with a bounded obstacle (Theorem~\ref{bddexistence}), and show that if the obstacle is continuous, so is the solution (Theorem~\ref{obstaclecontinuity}).

\section{Preliminaries}\label{preliminaries}

Let $n\geq 2$ and $\Omega_T = \Omega\times (0,T) \subset \R^n\times\R$, where $\Omega$ is a bounded domain. Consider the equation 
\begin{equation}\label{general.equation}
\partial_tu-\divergence \A(Du) =0\qquad\text{in }\qquad\Omega_T
\end{equation}
where $\A:\R^n\to\R^n$ is a $C^1$ vector field satisfying
\begin{equation}\label{assumptionsO}
\begin{cases}
\displaystyle{\langle D\A(\xi)\zeta, \zeta\rangle \geq \nu \frac{g(|\xi|)}{|\xi|}|\zeta|^2}\\[3mm]
\displaystyle{|D\A(\xi)| \leq L\,\frac{g(|\xi|)}{|\xi|}}\\[2mm]
\end{cases},
\end{equation}
for every $\xi\in\R^n\setminus\{0\}, \zeta\in\R^n$ and with structural constants $0<\nu\leq1\leq L$. We may assume without loss of generality that $\A(0)=0$ by replacing $\A(\xi)$ with $\A(\xi)-\A(0)$. The function $g:\R_+\to\R_+$ is assumed to be $C^1$-regular and to satisfy \eqref{O.property}.  Moreover, without loss of generality we may assume that
\begin{equation}\label{normalization}
 \int_0^1 g(\rho)\,d\rho = 1
\end{equation}
by scaling $g$ with a suitable constant and changing the structural constants accordingly.

\begin{Remark}\label{more.general.conditions}
Our results hold for a wider class of operators $\A$, which allow the presence of a function $g$ that is not $C^1$ but merely Lipschitz. Indeed, we may consider Lipschitz functions $g:\R_+\to\R_+$ satisfying \eqref{O.property} almost everywhere and vector fields $\A:\R^n\to\R^n$ in $W^{1,\infty}(\R^n)$ satisfying the monotonicity and Lipschitz assumptions
\begin{equation}\label{assumptionsO.weak}
\begin{cases}
\displaystyle{\langle \A(\xi_1)-\A(\xi_2),\xi_1-\xi_2\rangle \geq \nu \frac{g(|\xi_1|+|\xi_2|)}{|\xi_1|+|\xi_2|}|\xi_1-\xi_2|^2}\\[3mm]
\displaystyle{|\A(\xi_1)-\A(\xi_2)| \leq L\, \frac{g(|\xi_1|+|\xi_2|)}{|\xi_1|+|\xi_2|}|\xi_1-\xi_2|},
\end{cases},
\end{equation}
for every $\xi_1,\xi_2\in\R^n$ such that $|\xi_1|+|\xi_2|\neq 0$ and for some $0<\nu\leq 1\leq L$. 
\end{Remark}

\subsection{Notation}

We denote by $c$ a general constant {\em always larger than or equal to one}, possibly varying from line to line; relevant dependencies on parameters will be emphasized using parentheses, i.e., ~$c\equiv c(n,g_0,g_1)$ means that $c$  depends on $n,g_0,g_1$. 

We denote by $$ B_R(x_0):=\{x \in \R^n \, : \,  |x-x_0|< R\}$$ the open ball with center $x_0$ and radius $R>0$; when clear from the context or otherwise not important, we shall omit the center. Unless otherwise explicitly stated, different balls and cylinders in the same context will have the same center. The parabolic boundary of a cylindrical domain $\mathcal K=\mathcal D\times (t_1,t_2)\subset\R^n\times\R$ is defined as
\[
\partial_p\mathcal K:=\big(\overline {\mathcal D}\times \{t_1\}\big)\cup\big(\partial \mathcal D\times (t_1,t_2)\big).
\]
Naturally, the parabolic closure of $\mathcal K$ is then ${\overline{\mathcal K}}^p:=\mathcal K\cup\partial_p\mathcal K$. Accordingly with the customary use in the parabolic setting, when considering a sub-cylinder $\mathcal K$ (as above) compactly contained in $\Omega_T$, we shall mean that $\mathcal D\Subset\Omega$ and $0<t_1<t_2\leq T$; we will write in this case $\mathcal K\Subset\Omega_T$.

With $\mathcal B \subset \R^\ell$ being a measurable set, $\chi_{\mathcal B}$ denotes its characteristic function. If furthermore $\mathcal B$ has positive and finite measure and $f: \mathcal B \to \R^{k}$ is a measurable map, we shall denote by  
$$
   \mint_{\mathcal B}  f(y) \,dy  := \frac{1}{|\mathcal B|}\int_{\mathcal B}  f(y) \,dy
$$
the integral average of $f$ over $\mathcal B$. We shall also as usual denote 
\[
\essosc_{\mathcal B}f:=\esssup_{\mathcal B}f-\essinf_{\mathcal B}f.
\]
By $q^*$ we denote the Sobolev conjugate exponent of $1\leq q<n$, i.e.,
\begin{equation}\label{convention}
  q^*:=\displaystyle{\frac{nq}{n-q}} 
\end{equation}
 With $s$ being a real number, we denote $s_+:=\max\{s, 0\}$ and $s_-:=\max\{-s, 0\}$. Finally, $\R_+:=[0,\infty)$, $\N$ is the set $\{1,2,\dots\}$ and $\N_0=\N\cup\{0\}$.

\subsection{Properties of $g$ and basic inequalities}

We begin by collecting useful properties of the function $g$ and some basic inequalities that will be needed later. For proofs see for example \cite{Adams} and \cite{DE}. 

First, observe that $g$ is strictly increasing and satisfies $g(0)=0$ and $\lim_{s\to\infty}g(s)=\infty$. Define the function $G: \R_+\to\R_+$ as 
\begin{equation}\label{G}
 G(s) := \int_0^s g(r)\,dr
\end{equation}
and its Young complement as
\[
 \widetilde G(s) := \sup_{r\geq 0}\,\big(sr - G(r)\big).
\]
The functions $G$ and $\widetilde G$ are strictly increasing, strictly convex, and map zero to zero, in particular, they are Young functions. Moreover, they both satisfy an Orlicz-type condition, namely for $s>0$ 
\begin{equation}\label{O.property3}
 g_0 \leq \frac{sG'(s)}{G(s)} \leq g_1 \qquad\textrm{and}\qquad  \frac{g_1}{g_1-1} \leq \frac{s\widetilde G'(s)}{\widetilde G(s)} \leq \frac{g_0}{g_0-1}.
\end{equation}
For $g$ we have the so-called $\Delta_2$-condition 
\begin{equation}\label{gDelta2}
 \min\left\{\alpha^{g_0-1},\alpha^{g_1-1}\right\}g(s)\leq g(\alpha s)\leq \max\left\{\alpha^{g_0-1},\alpha^{g_1-1}\right\}g(s),
\end{equation}
and corresponding inequalities hold also for $G$ and $\widetilde G$. Moreover, $G$ satisfies the triangle inequality modulo a constant 
\begin{equation}\label{triangle}
 G(s+r)\leq 2^{g_1}\big(G(s)+G(r)\big), \qquad s,r\geq 0,
\end{equation}
and the following important inequality
\begin{equation}\label{Yf}
 \widetilde G\left(\frac{G(s)}{s}\right) \leq G(s), \qquad s>0.
\end{equation}
Finally, for $\eps\in(0,1]$ we have the Young's inequality with $\eps$
\[
 sr \leq \varepsilon G(s) + \varepsilon^{-\frac{1}{g_0-1}}\widetilde G(r), \qquad s,r\geq 0.
\]

By writing 
\[
 \A(\xi)=\int_0^1 D\A(s\xi)\xi\ds
\]
we easily see that assumptions \eqref{assumptionsO} imply 
\begin{equation}\label{assumptions1}
\begin{cases}
\displaystyle{\left\langle\A(\xi),\xi\right\rangle \geq \tilde\nu G(|\xi|)}\\[3mm]
\displaystyle{|\A(\xi)| \leq \tilde L\,\frac{G(|\xi|)}{|\xi|}}\\[2mm]
\end{cases}
\end{equation}
for $\xi\in\R^n$ with $\tilde\nu:=\frac{g_0}{g_1-1}\nu$ and $\tilde L:=\frac{g_1}{g_0-1}L$, and in the case $G(s)=s^p$ these are precisely the commonly used assumptions in the study of $p$-Laplace type equations. 

\begin{Lemma}\label{monotonicity}
 \emph{(Strict monotonicity)} There exists a constant $c\equiv c(g_0,g_1,\nu)$ such that 
 \[
  \langle \A(\xi_1)-\A(\xi_2),\xi_1-\xi_2 \rangle\geq cg'(|\xi_1|+|\xi_2|)|\xi_1-\xi_2|^2
 \]
 for every $\xi_1,\xi_2\in\R^n$. 
\end{Lemma}

\begin{Remark}
 Note that since $g$ is increasing, Lemma~\ref{monotonicity} implies 
 \[
  \langle \A(\xi_1)-\A(\xi_2),\xi_1-\xi_2 \rangle\geq 0
 \]
for every $\xi_1,\xi_2\in\R^n$.
\end{Remark}

Define the natural quantity $V_g:\R^n\to\R^n$ by 
\[
 V_g(\xi) = \left(\frac{g(|\xi|)}{|\xi|}\right)^{\frac{1}{2}}\xi
\]
when $\xi\neq 0$ and set $V_g(0)=0$. Clearly $V_g$ is continuous and, moreover, has a continuous inverse by the inverse function theorem. 

\begin{Lemma}\label{Vglemma}
 There exists a constant $c\equiv c(g_0,g_1)$ such that
 \[
  |V_g(\xi_1)-V_g(\xi_2)|^2 \leq cg'(|\xi_1|+|\xi_2|)|\xi_1-\xi_2|^2
 \]
for every $\xi_1,\xi_2\in\R^n$.
\end{Lemma}

\subsection{Orlicz spaces}

Let $G$ be as in \eqref{G}. A measurable function $u:\Omega\to\R$ belongs to the Orlicz space $L^G(\Omega)$ if it satisfies 
\[
 \int_{\Omega}G(|u|)\dx < \infty.
\]
The space $L^G(\Omega)$ is a vector space, since $G$ satisfies the $\Delta_2$-condition, and it can be shown to be a 
Banach space if endowed with the Luxemburg norm 
\[
||u||_{L^G(\Omega)} := \inf\left\{\lambda>0:\int_{\Omega}G\bigg(\frac{|u|}{\lambda}\bigg)\dx\leq 1\right\}.
\]
A function $u$ belongs to $L_{\textrm{loc}}^G(\Omega)$, if $u\in L^G(\Omega')$ for every $\Omega'\Subset\Omega$. If also the weak gradient 
of $u$ belongs to $L^G(\Omega)$, we say that $u\in W^{1,G}(\Omega)$. The corresponding space with zero boundary values, denoted $W^{1,G}_0(\Omega)$, is the completion of $C^{\infty}_c(\Omega)$ under the norm 
\[
||u||_{W^{1,G}(\Omega)}:=||u||_{L^G(\Omega)}+||Du||_{L^G(\Omega)}. 
\]
We denote by $V^G(\Omega_T)$ the space of functions $u\in L^G(\Omega_T)\cap L^1(0,T;W^{1,1}(\Omega))$ for which also the weak spatial gradient $Du$ belongs to $L^G(\Omega_T)$. The space $V^G(\Omega_T)$ is also a Banach space with the norm 
\[
 ||u||_{V^G(\Omega_T)} := ||u||_{L^G(\Omega_T)}+||Du||_{L^G(\Omega_T)}.
\]
Moreover, we denote by $V_0^G(\Omega_T)$ the space of functions $u\in V^G(\Omega_T)$ for which $u(\cdot,t)$ belongs to $W^{1,G}_0(\Omega)$ for almost every $t\in(0,T)$, while the localized version $V^G_\loc(\Omega_T)$ is defined, as above, in the customary way. We also use the shorthand notation 
\[
V^{2,G}(\Omega_T):=L^\infty\left(0,T;L^2(\Omega)\right)\cap V^G(\Omega_T)
\]
and similarly for the localized and the zero trace versions. More on Orlicz spaces can be found, for example, in \cite{Adams}.

We gather here some useful results for Orlicz space functions. 

\begin{Lemma}
 \emph{(H\"older's inequality)} Let $u\in L^G(\Omega_T)$ and $v\in L^{\widetilde G}(\Omega_T)$. Then $uv\in L^1(\Omega_T)$ and we have 
 \[
  \int_{\Omega_T}|u||v|\dx\dt \leq 2||u||_{L^G(\Omega_T)}||v||_{L^{\widetilde G}(\Omega_T)}.
 \]
\end{Lemma}
\begin{proof}
See for example \cite{Adams}.
\end{proof}

\begin{Lemma}\label{characteristicnorm}
 The inequality 
 \[
  ||\chi_E||_{L^G(\Omega_T)}\leq \max\left\{|E|^{\frac{1}{g_1}},|E|^{\frac{1}{g_0}}\right\}.
 \]
 holds for every $E\subset\Omega_T$.
\end{Lemma}
\begin{proof}
We may assume $|E|>0$, since the claim trivially holds if $|E|=0$. It is easy to show that $G^{-1}(s) \geq \min\left\{s^{\frac{1}{g_1}},s^{\frac{1}{g_0}}\right\}$ for every $s\geq 0$. Now, since 
\[
 \int_{\Omega_T}G\left(G^{-1}\left(\frac{1}{|E|}\right)\chi_E\right)\dx\dt=1,
\]
we have 
\[
 ||\chi_E||_{L^G(\Omega_T)}=\frac{1}{G^{-1}\left(1/|E|\right)}
 \leq \max\left\{|E|^{\frac{1}{g_1}},|E|^{\frac{1}{g_0}}\right\}.\qedhere
\]
\end{proof}

\begin{Lemma}\label{poincare}
 \emph{(Poincar\'e's inequality)} Let $\Omega\subset\R^n$ be a bounded set. Then 
 \[
  \int_{\Omega}G\left(\frac{|u|}{\diam(\Omega)}\right)\dx\leq \int_{\Omega}G(|Du|)\dx
 \]
for every $u\in W^{1,G}_0(\Omega)$.
\end{Lemma}

\begin{proof}
See Lemma 2.2 in \cite{Lieberman1}.
\end{proof}

\begin{Lemma}\label{parabolicsobolev}
 \emph{(Parabolic Sobolev's inequality)} Let $1\leq q<\min\{n,g_0\}$ and $B_R\times\Gamma\subset\R^{n+1}$. Then there exists a constant $c\equiv c(n,g_1,q)$ such that 
 \begin{multline*}
  \mint_{B_R\times\Gamma} G\left(\frac{|u|}{R}\right)^{1/q}|u|^{2(1-1/q^*)}\dx\dt\\
  \leq c\esssup_\Gamma \left(\mint_{B_R}|u|^2\dx\right)^{1-1/q^*}\left(\mint_{B_R\times\Gamma}G(|Du|)\dx\dt\right)^{1/q}
 \end{multline*}
for every $u\in V^{2,G}_0(B_R\times\Gamma)$.
\end{Lemma}
\begin{proof}
Set $H(s):=G(s^{1/q})$ and $F(s):= \widetilde H(s^q)^{1/q}$. Simple calculations show that 
\[
 \frac{g_0}{q}\leq\frac{sH'(s)}{H(s)}\leq\frac{g_1}{q} \qquad\textrm{and}\qquad\frac{g_1}{g_1-q}\leq\frac{sF'(s)}{F(s)}\leq\frac{g_0}{g_0-q}
\]
for every $s > 0$. Moreover, the elementary inequality $a^q-b^q\geq(a-b)^q, a\geq b\geq 0,$ yields 
\begin{align*}
 F(s)&=\left(\sup_{r>0}\left(s^qr-G(r^{1/q})\right)\right)^{1/q}=\sup_{r>0}\big((sr)^q-G(r)\big)^{1/q}\\
 &= \sup_{\substack{r>0 \\ (sr)^q\geq G(r)}}\big((sr)^q-G(r)\big)^{1/q} \geq \sup_{\substack{r>0 \\ (sr)^q\geq G(r)}}\big(sr-G(r)^{1/q}\big)\\
 &=\sup_{r>0}\big(sr-G(r)^{1/q}\big),
\end{align*}
and thus $\widetilde F(s)\leq G(s)^{1/q}$ due to the fact that for Young functions $A$ and $B$, $A\leq B$ implies $\widetilde B\leq \widetilde A$. Now for almost every $t\in\Gamma$ we have 
\begin{align*}
 \left|D\widetilde F\left(\frac{|u(\cdot,t)|}{R}\right)\right|^q &= \widetilde F'\left(\frac{|u(\cdot,t)|}{R}\right)^q\left(\frac{|Du(\cdot,t)|}{R}\right)^q\\
 &\leq \frac{c(g_1,q)}{R^q}\bigg(\frac{\widetilde F(|u(\cdot,t)|/R)}{|u(\cdot,t)|/R}\bigg)^q|Du(\cdot,t)|^q\\
 &\leq \frac{c(g_1,q)}{R^q}\left[\eps\widetilde H\left(\bigg(\frac{\widetilde F(|u(\cdot,t)|/R)}{|u(\cdot,t)|/R}\bigg)^q\right) + \eps^{1-g_1/q} H\big(|Du(\cdot,t)|^q\big)\right]\\
 &=\frac{c(g_1,q)}{R^q}\left[\eps F\bigg(\frac{\widetilde F(|u(\cdot,t)|/R)}{|u(\cdot,t)|/R}\bigg)^q + \eps^{1-g_1/q} G(|Du(\cdot,t)|)\right]\\
 &\leq \frac{c(g_1,q)}{R^q}\left[\eps \widetilde F\left(\frac{|u(\cdot,t)|}{R}\right)^q + \eps^{1-g_1/q} G(|Du(\cdot,t)|)\right]
\end{align*}
by Young's inequality with $\eps\in(0,1)$, the definitions of $F$ and $H$, and \eqref{Yf}. This implies $\widetilde F\left(|u(\cdot,t)|/R\right) \in W^{1,q}_0(B_R)$ for almost every $t\in\Gamma$, since $u\in V^G_0(B_R\times\Gamma)$. Therefore we may apply the elliptic Sobolev's inequality to obtain 
\begin{align*}
 &\left(\mint_{B_R} \widetilde F\left(\frac{|u(\cdot,t)|}{R}\right)^{q^*}\dx\right)^{q/q^*}\leq c(n,q)R^q\mint_{B_R}\left|D\widetilde F\left(\frac{|u(\cdot,t)|}{R}\right)\right|^q\dx\\ \notag
 &\qquad\leq c(n,g_1,q)\mint_{B_R}\left[\eps \widetilde F\left(\frac{|u(\cdot,t)|}{R}\right)^q + \eps^{1-g_1/q} G(|Du(\cdot,t)|)\right]\dx\\
 &\qquad\leq \frac12 \left(\mint_{B_R} \widetilde F\left(\frac{|u(\cdot,t)|}{R}\right)^{q^*}\dx\right)^{q/q^*} + c(n,g_1,q) \mint_{B_R}G(|Du(\cdot,t)|)\dx,\\
 \end{align*}
where we also used H\"older's inequality and chose $\eps \equiv \eps(n,g_1,q)$ small enough. Hence 
\begin{equation}\label{ellipticsobolev}
 \left(\mint_{B_R} \widetilde F\left(\frac{|u(\cdot,t)|}{R}\right)^{q^*}\dx\right)^{1/q^*}\leq c\left(\mint_{B_R} G(|Du(\cdot,t)|)\dx\right)^{1/q}
\end{equation}
for almost every $t\in\Gamma$, where $c\equiv c(n,g_1,q)$. 

Using another elementary inequality $a^q-2^qb^q\leq 2^q(a-b)^q, a\geq b\geq 0,$ gives 
\begin{align*}
 F(s)&=\sup_{r>0}\big((sr)^q-G(r)\big)^{1/q}=\sup_{\substack{r>0 \\ (sr)^q\geq G(r)}}\left((sr)^q-2^q\left(\frac12 G(r)^{1/q}\right)^q\right)^{1/q}\\
 &\leq 2\sup_{\substack{r>0 \\ (2sr)^q\geq G(r)}}\left(sr-\frac12 G(r)^{1/q}\right)
 =\sup_{r>0}\left(sr-G\left(\frac r2\right)^{1/q}\right),
\end{align*}
and therefore 
\[
 \widetilde F(s) \geq G\left(\frac s2\right)^{1/q} \geq 2^{-g_1/q}G(s)^{1/q}.
\]
By combining this with H\"older's inequality and \eqref{ellipticsobolev} we finally obtain 
\begin{align*}
 \mint_{B_R\times\Gamma} G\left(\frac{|u|}{R}\right)^{1/q}&|u|^{2(1-1/q^*)}\dx\dt
 \leq c\mint_{B_R\times\Gamma} \widetilde F\left(\frac{|u|}{R}\right)|u|^{2(1-1/q^*)}\dx\dt\\
 &\leq c\mint_\Gamma\left(\mint_{B_R}\widetilde F\left(\frac{|u|}{R}\right)^{q^*}\dx\right)^{1/q^*}\left(\mint_ {B_R}|u|^2\dx\right)^{1-1/q^*}\dt\\
 &\leq c\esssup_\Gamma \left(\mint_ {B_R}|u|^2\dx\right)^{1-1/q^*}\left(\mint_{B_R\times\Gamma}G(|Du|)\dx\dt\right)^{1/q}.\qedhere
\end{align*}
\end{proof}
\begin{Remark}
 Since $g_0>2n/(n+2)$ we may take $q=2n/(n+2)$ in the previous Lemma, which yields 
\begin{multline*}
  \mint_{B_R\times\Gamma} G\left(\frac{|u|}{R}\right)^{1/2+1/n}|u|\dx\dt\\
  \leq c\esssup_\Gamma \left(\mint_{B_R}|u|^2\dx\right)^{1/2}\left(\mint_{B_R\times\Gamma}G(|Du|)\dx\dt\right)^{1/2+1/n}
 \end{multline*}
 for every $u\in V^{2,G}_0(B_R\times\Gamma)$.
\end{Remark}

\section{Useful results for solutions}\label{usefulresults}

In this section we collect and partly prove various results for weak solutions to \eqref{eq} that are standard for the evolutionary $p$-Laplace equation. The main result of the section is Theorem~\ref{superlimit}, which states that the limit of a uniformly bounded sequence of weak supersolutions is also a weak supersolution. 

We begin with the definition of weak solutions. 

\begin{Definition}
 A function $u$ is a \emph{weak solution} in $\Omega_T$, if $u\in V_{\textrm{loc}}^{2,G}(\Omega_T)$ and it satisfies 
\begin{equation}\label{weaksolution}
 -\int_{\Omega_T}u\,\partial_t\eta\dx\dt + \int_{\Omega_T}\A(Du)\cdot D\eta\dx\dt = 0
\end{equation}
for every $\eta \in C^{\infty}_c(\Omega_T)$. If instead of equality we have $\geq(\leq)$ for every nonnegative 
$\eta \in C^{\infty}_c(\Omega_T)$, we say that $u$ is a \emph{weak supersolution (subsolution)} in $\Omega_T$. 
\end{Definition}
\begin{Remark}\label{minussuperissub}
 If $u$ is a weak supersolution, it is easy to see that $-u$ is a weak subsolution to the same equation with $\A(Du)$ replaced by 
 $-\A(-Du)$. However, since the latter equation also satisfies the structural conditions \eqref{assumptionsO}, we may 
 assume without loss of generality that $\A(\xi) = -\A(-\xi)$. Therefore, if $u$ is a weak supersolution, then $-u$ is a weak subsolution.
\end{Remark}

The following Caccioppoli inequality is proven in \cite{BL}.

\begin{Lemma}[Caccioppoli inequality]\label{caccioppoli}
Let ${\mathcal K}:=\mathcal D\times(t_1,t_2)\Subset\Omega_T$ and let $u$ be a weak solution in ${\mathcal K}$. Then there exists a constant $c\equiv c(g_0,g_1,\nu,L)$ such that 
 \begin{multline*}
  \esssup_{(t_1,t_2)}\int_{\mathcal D}(u-k)_\pm^2\varphi^{g_1}\dx + \int_{\mathcal K}G\big(|D(u-k)_\pm|\big)\varphi^{g_1}\dx\dt\\
 \leq \int_{\mathcal D}\big[(u-k)_\pm^2\varphi^{g_1}\big](\cdot,t_1)\,dx+ c\int_{\mathcal K}\Big[G\big(|D\varphi|(u-k)_\pm\big)+(u-k)_\pm^2\left|\partial_t\varphi\right|\Big]\dx\dt
 \end{multline*}
for any $k\in\R$ and for every $\varphi \in W^{1,\infty}({\mathcal K})$ vanishing in a neighborhood of $\partial \mathcal D\times(t_1,t_2)$ and with $0\leq\varphi\leq 1$. The same inequality but only with the ``$+$'' sign holds for weak subsolutions.
\end{Lemma}

For the following comparison principle we add the extra assumption that the functions in question are continuous. This weaker version is sufficient for our purposes. Again, the proof can be found in \cite{BL}. 

\begin{Proposition}\label{comparisonprinciple}
 \emph{(Comparison principle)} Let $\mathcal K := \mathcal D\times(t_1,t_2)\subset\Omega_T$ and let $u\in C^0({\overline {\mathcal K}}^p)$ be a weak subsolution and $v\in C^0({\overline {\mathcal K}}^p)$ a weak supersolution in $\mathcal K$. If $u\leq v$ on $\partial_p \mathcal K$, then $u\leq v$ in $\overline {\mathcal K}^p$.
\end{Proposition}

We obtain the maximum principle as an easy corollary. 

\begin{Corollary}\label{maximumprinciple}
 \emph{(Maximum principle)} Let $\mathcal K\subset\Omega_T$ and let $u\in C^0(\overline {\mathcal K}^p)$ be a weak solution in $\mathcal K$. Then 
 \begin{equation}\label{max1}
  \inf_{\partial_p\mathcal K}u\leq u\leq\sup_{\partial_p\mathcal K}u
 \end{equation}
in $\overline {\mathcal K}^p$ and, moreover, 
 \begin{equation}\label{max2}
  \sup_{\overline {\mathcal K}^p}|u|=\sup_{\partial_p\mathcal K}|u|.
 \end{equation}
\end{Corollary}

The next pasting lemma states that if we replace a supersolution by a smaller supersolution in some part of the cylinder such that they coincide on the boundary, the resulting function is still a supersolution. Again, we assume that the functions are continuous. The proof follows ideas used in~\cite{KKP}. 

\begin{Lemma}\label{superglued}
 Let $Q_1:=K_1\times(t_1,\tau_1), Q_2:=K_2\times(t_2,\tau_2)\subset\Omega_T$ be such that $\tau_1\leq\tau_2$, and let $v_1\in C^0(\overline Q_1^p)$ and 
 $v_2\in C^0(\overline Q_2^p)$ be weak supersolutions (subsolutions) in $Q_1$ and $Q_2$, respectively. If $v_1\geq v_2$ 
 ($v_1\leq v_2$) in $Q_1\cap Q_2$ and $v_1=v_2$ on $Q_1\cap\partial_pQ_2$, then 
 \[
  v = \left\{ \begin{array}{ll}
v_1 & \textrm{in}\,\, Q_1\setminus Q_2\\
v_2 & \textrm{in}\,\, Q_1\cap Q_2\\
\end{array}\right.
 \]
is a weak supersolution (subsolution) in $Q_1$.
\end{Lemma}
\begin{proof}
Fix a nonnegative $\varphi\in C^{\infty}_c(Q_1)$, $\varepsilon>0$, and define 
\[
 w_{\varepsilon} := \left\{ \begin{array}{lll}
1, & v_1 \geq v_2+2\varepsilon\\[2mm]
\frac{v_1-v_2-\varepsilon}{\varepsilon}, & v_2+\varepsilon \leq v_1 < v_2 + 2\varepsilon\\[2mm]
0, & v_1 < v_2 + \varepsilon
\end{array}\right..
\]
in $Q_1\cap Q_2$ and $w_{\varepsilon}=0$ in $Q_1\setminus Q_2$. The test functions 
\[
 \eta_1=(1- w_{\varepsilon} )\varphi\qquad\textrm{and}\qquad\eta_2= w_{\varepsilon} \varphi
\]
have compact support in $Q_1$ and $Q_1\cap Q_2$, respectively, and can thus be regularized using mollification; we shall proceed formally. Now, assuming $v_1$ and $v_2$ are weak supersolutions, summing up their respective weak formulations we obtain 
\[
 \begin{split}
 0&\leq-\int_{Q_1} v_1 \partial_t\big((1- w_{\varepsilon} )\varphi\big)\dx\dt 
 +\int_{Q_1} \A(Dv_1) \cdot D\big((1- w_{\varepsilon} )\varphi\big)\dx\dt\\
 &\quad-\int_{Q_1} v_2 \partial_t( w_{\varepsilon} \varphi)\dx\dt 
 +\int_{Q_1} \A(Dv_2) \cdot D( w_{\varepsilon} \varphi)\dx\dt\\
 &=-\int_{Q_1}\big( v_1 (1- w_{\varepsilon} )+ v_2  w_{\varepsilon} \big)\partial_t\varphi\dx\dt
 +\int_{Q_1} (v_1-v_2) \partial_t w_{\varepsilon} \varphi\dx\dt\\
 &\quad\quad+\int_{Q_1}\big( \A(Dv_1) (1- w_{\varepsilon} )+ \A(Dv_2)  w_{\varepsilon} 
 \big)\cdot D\varphi\dx\dt\\
 &\quad\quad\quad-\int_{Q_1} \big(\A(Dv_1)-\A(Dv_2)\big) \cdot D w_{\varepsilon} \varphi\dx\dt.
 \end{split}
\]
The monotonicity of $\A$ yields 
\[
 \begin{split}
 &\int_{Q_1} \big(\A(Dv_1)-\A(Dv_2)\big) \cdot D w_{\varepsilon} \varphi\dx\dt\\
 &\qquad=\frac{1}{\varepsilon}\int_{Q_1}\big(\A(Dv_1)-\A(Dv_2)\big)
 \cdot(Dv_1-Dv_2)\chi_{\{v_2+\varepsilon\leq v_1<v_2+2\varepsilon\}}\varphi\dx\dt\geq 0,
 \end{split}
\]
and by integration by parts we get 
\[
\begin{split}
 \int_{Q_1}& (v_1-v_2) \partial_t w_{\varepsilon} \varphi\dx\dt=\eps\int_{Q_1}(1+w_\eps)\partial_tw_\eps\varphi\dx\dt\\
 &=-\frac{\varepsilon}{2}\int_{Q_1}(1+w_\eps)^2\partial_t\varphi\dx\dt\leq 2\varepsilon\|\partial_t\varphi\|_{L^\infty(Q_1)}|Q_1|,
 \end{split}
\]
since $w_\eps\leq 1$. Therefore, 
\[
 \begin{split}
  0&\leq-\int_{Q_1}\big(v_1(1-w_{\varepsilon})+v_2w_{\varepsilon}\big)\partial_t\varphi\dx\dt
  +2\varepsilon\|\partial_t\varphi\|_{L^\infty(Q_1)}|Q_1|\\
  &\quad+\int_{Q_1}\big(\A(Dv_1)(1-w_{\varepsilon})+\A(Dv_2)w_{\varepsilon}\big)\cdot D\varphi\dx\dt.
 \end{split}
\]
Since $w_{\varepsilon}=0$ in $Q_1\setminus Q_2$ and $v_1\geq v_2$ in $Q_1\cap Q_2$, we have 
\[
 \left|w_{\varepsilon}-\chi_{Q_1\cap Q_2}\right|
 \leq \chi_{Q_1\cap Q_2\cap\{v_2\leq v_1<v_2 + 2\varepsilon\}}\to 0
\]
as $\varepsilon\to 0$, and thus, we obtain 
\[
 \begin{split}
  0&\leq-\lim_{\varepsilon\to 0}\int_{Q_1}\big(v_1(1-w_{\varepsilon})+v_2w_{\varepsilon}\big)\partial_t\varphi\dx\dt\\
  &\quad+\lim_{\varepsilon\to 0}\int_{Q_1}\big(\A(Dv_1)(1-w_{\varepsilon})+\A(Dv_2)w_{\varepsilon}\big)\cdot D\varphi\dx\dt\\
  &=-\int_{Q_1}\left(v_1\chi_{Q_1\setminus Q_2}+v_2\chi_{Q_1\cap Q_2}\right)\partial_t\varphi\dx\dt\\
  &\quad+\int_{Q_1}\left(\A(Dv_1)\chi_{Q_1\setminus Q_2}+\A(Dv_2)\chi_{Q_1\cap Q_2}\right)\cdot D\varphi\dx\dt\\
  &=-\int_{Q_1} v\,\partial_t\varphi\dx\dt+\int_{Q_1}\A(Dv)\cdot D\varphi\dx\dt,
 \end{split}
\]
showing that $v$ is a weak supersolution in $Q_1$. 

If $v_1$ and $v_2$ are assumed to be weak subsolutions such that $v_1\leq v_2$, then by applying the above reasoning to $-v_1$ and $-v_2$ 
we see that $-v$ is a weak supersolution, and the result follows. 
\end{proof}

The following lemma can be proved in a very similar fashion as the previous one.

\begin{Lemma}\label{minsuper}
 Let $u$ and $v$ be weak supersolutions (subsolutions) in $Q\subset\Omega_T$. Then also $\min\{u,v\}$ is a weak supersolution 
 ($\max\{u,v\}$ is a weak subsolution) in $Q$.
\end{Lemma}

\subsection{Convergence properties of supersolutions}\label{convergenceproperties}

We end the section by proving an important convergence result that is crucial in proving the existence of a solution to the obstacle problem. For this we need the following lemma. The proof follows the guidelines of Theorem~6 in~\cite{LM}, see also \cite{KKP,BDGO}. 

\begin{Lemma}\label{gradientconvergence}
 Let $Q\subset\Omega_T, M\geq 1$, and let $(u_i)_{i=1}^{\infty}$ be a sequence of weak supersolutions in $Q$ such that $|u_i|\leq M$ for every $i\in\N$ and $u_i\to u$ almost everywhere in $Q$. Then $Du_i\to Du$ almost everywhere in $Q$.
\end{Lemma}
\begin{proof}
 Let $Q'\Subset Q$ and choose $\widetilde Q$ such that $Q'\Subset\widetilde Q\Subset Q$. Let $\varphi\in C^{\infty}_c(Q), \tilde\varphi\in C^{\infty}_c(\widetilde Q)$ be such that $0\leq\varphi,\tilde\varphi\leq 1$, $\varphi=1$ in $\widetilde Q$, $\tilde\varphi=1$ in $Q'$, and 
\[
\left|\left|\partial_t\varphi\right|\right|_{L^{\infty}(Q)},\left|\left|\partial_t\tilde\varphi\right|\right|_{L^{\infty}(Q)},||D\varphi||_{L^{\infty}(Q)},||D\tilde\varphi||_{L^{\infty}(Q)}\leq C
\]
for some $C\geq 1$. Applying the Caccioppoli estimate, Lemma~\ref{caccioppoli}, to the nonnegative weak subsolution $M-u_i$ (with $k=0$) gives 
\[
 \begin{split}
 \int_{\widetilde Q}G(|Du_i|)\dx\dt &\leq c\int_{Q}(M-u_i)^2\left|\partial_t\varphi\right|\dx\dt\\
  &\qquad\qquad\qquad +c\int_{Q}G\big(|D\varphi|(M-u_i)\big)\dx\dt\\
  &\leq c\big(M^2C+M^{g_1}C^{g_1}\big)|Q|=: M_1.
 \end{split}
\]
Thus, the sequence $(Du_i)_{i=1}^{\infty}$ is uniformly bounded in $L^G(\widetilde Q)$. Moreover, the sequence 
$\left(\A(Du_i)\right)_{i=1}^{\infty}$ is uniformly bounded in $L^{\widetilde G}(\widetilde Q)$, since 
\[
 \begin{split}
 \int_{\widetilde Q}\widetilde G(|\A(Du_i)|)\dx\dt &\leq \int_{\widetilde Q}\widetilde G\left(\frac{g_1L}{g_0-1}\frac{G(|Du_i|)}{|Du_i|}\right)\dx\dt\\
 &\leq c\int_{\widetilde Q}G(|Du_i|)\dx\dt
 \leq cM_1=: M_2.
 \end{split}
\]
by $\eqref{assumptionsO}_2$ and \eqref{Yf}. Assuming without loss of generality that $M_2\geq 1$ it is then easy to see that also 
\[
 ||\A(Du_i)||_{L^{\widetilde G}(\widetilde Q)}\leq M_2.
\]

Denote for $j,k\in\N$ 
\[
 w_{jk} := \left\{ \begin{array}{lll}
\delta, & u_j-u_k>\delta\\[2mm]
u_j-u_k, &|u_j-u_k|\leq\delta\\[2mm]
-\delta, & u_j-u_k<-\delta
\end{array}\right.,
\]
where $\delta>0$. Choose 
\[
 \eta_{j}=(\delta-w_{jk})\tilde\varphi\qquad\textrm{and}\qquad\eta_{k}=(\delta+w_{jk})\tilde\varphi
\]
as the test functions in \eqref{weaksolution} for the weak supersolutions $u_j$ and $u_k$. Observe that $\eta_j$ and $\eta_k$ are nonnegative. These formal choices can be justified by standard regularization methods. Summing up the weak formulations yields 
\begin{equation}\label{A}
 \begin{split}
  0&\leq-\int_{Q}u_j\partial_t\eta_{j}\dx\dt + \int_{Q}\A(Du_j)\cdot D\eta_{j}\dx\dt\\
 &\quad-\int_{Q}u_k\partial_t\eta_{k}\dx\dt + \int_{Q}\A(Du_k)\cdot D\eta_{k}\dx\dt\\
 &=\int_{\widetilde Q}(u_j-u_k)\partial_tw_{jk}\,\tilde\varphi\dx\dt+\int_{\widetilde Q}(u_j-u_k)w_{jk}\partial_t\tilde\varphi\dx\dt\\
 &\quad-\delta\int_{\widetilde Q}(u_j+u_k)\partial_t\tilde\varphi\dx\dt-\int_{\widetilde Q}\big(\A(Du_j)-\A(Du_k)\big)\cdot Dw_{jk}\,\tilde\varphi\dx\dt\\
 &\quad-\int_{\widetilde Q}\big(\A(Du_j)-\A(Du_k)\big)\cdot D\tilde\varphi\,w_{jk}\dx\dt\\
 &\quad+\delta\int_{\widetilde Q}\big(\A(Du_j)+\A(Du_k)\big)\cdot D\tilde\varphi\dx\dt.
 \end{split}
\end{equation}
Integration by parts gives 
\begin{align*}
 \int_{\widetilde Q}&(u_j-u_k)\partial_tw_{jk}\,\tilde\varphi\dx\dt=-\int_{\widetilde Q}\int_{-2M}^{u_j-u_k}s\chi_{\{|s|\leq\delta\}}\ds\,\partial_t\tilde\varphi\dx\dt\\
 &\leq \int_{\widetilde Q}\int_{-2M}^{u_j-u_k}|s|\chi_{\{|s|\leq\delta\}}\ds|\partial_t\tilde\varphi|\dx\dt
 \leq 4MC|Q|\delta.
\end{align*}
Since $|w_{jk}|\leq\delta$, we can estimate the second and third term by $2MC|Q|\delta$. By H\"older's inequality and Lemma~\ref{characteristicnorm} we obtain 
\[
 \begin{split}
 &-\int_{\widetilde Q}\big(\A(Du_j)-\A(Du_k)\big)\cdot D\tilde\varphi\,w_{jk}\dx\dt \leq C\delta\int_{\widetilde Q}|\A(Du_j)-\A(Du_k)|\dx\dt\\
 &\qquad\qquad\qquad\qquad\qquad\leq 2C\delta||\chi_{\widetilde Q}||_{L^G(\widetilde Q)}||\A(Du_j)-\A(Du_k)||_{L^{\widetilde G}(\widetilde Q)}\\
 &\qquad\qquad\qquad\qquad\qquad\leq 4M_2C\max\left\{|Q|^{\frac{1}{g_1}},|Q|^{\frac{1}{g_0}}\right\}\delta.
 \end{split}
\]
Exactly the same estimate holds also for the last term in \eqref{A}. Thus, we have 
\[
 \int_{Q_{\delta}}(\A(Du_j)-\A(Du_k))\cdot (Du_j-Du_k)\,\tilde\varphi\dx\dt \leq c\delta,
\]
where $Q_{\delta}:=\widetilde Q\cap\{|u_j-u_k|\leq\delta\}$ and $c$ depends on $g_0,g_1,\nu,L,M,C$, and $|Q|$ but not on $j$ and $k$. 

Next by Lemma~\ref{Vglemma} and Lemma~\ref{monotonicity} we obtain 
\[
 \begin{split}
 \int_{Q_{\delta}}&|V_g(Du_j)-V_g(Du_k)|^2\tilde\varphi\dx\dt \leq c\int_{Q_{\delta}}g'(|Du_j|+|Du_k|)|Du_j-Du_k|^2\tilde\varphi\dx\dt\\
 &\qquad\leq c\int_{Q_{\delta}}(\A(Du_j)-\A(Du_k))\cdot (Du_j-Du_k)\tilde\varphi\dx\dt\leq c\delta.
 \end{split}
\]
Using H\"older's inequality and the fact that 
\[
\int_Q|V_g(Du_i)|^2\dx\dt\leq g_1\int_QG(|Du_i|)\dx\dt\leq g_1M_1 
\]
for all $i\in\N$ gives 
 \begin{align*}
  \int_{Q'}&|V_g(Du_j)-V_g(Du_k)|\dx\dt\\
  &\leq\int_{Q_{\delta}}|V_g(Du_j)-V_g(Du_k)|\tilde\varphi\dx\dt+\int_{\widetilde Q\setminus Q_{\delta}}|V_g(Du_j)-V_g(Du_k)|\tilde\varphi\dx\dt\\
  &\leq|Q|^{\frac{1}{2}}\left(\int_{Q_{\delta}}|V_g(Du_j)-V_g(Du_k)|^2\tilde\varphi\dx\dt\right)^{\frac{1}{2}}\\
  &\qquad\qquad\qquad\qquad+|\widetilde Q\setminus Q_{\delta}|^{\frac{1}{2}}
  \left(\int_{Q}|V_g(Du_j)-V_g(Du_k)|^2\dx\dt\right)^{\frac{1}{2}}\\
  &\leq c\delta^{\frac{1}{2}}+c|\widetilde Q\setminus Q_{\delta}|^{\frac{1}{2}}.
 \end{align*}
Fix $\varepsilon>0$ and choose $\delta$ such that $c\delta^{\frac{1}{2}}<\frac{\varepsilon}{2}$ holds. Since the sequence 
$(u_i)_{i=1}^{\infty}$ converges almost everywhere, and therefore in measure, we may choose $j$ and $k$ large enough such that $c|Q\setminus Q_{\delta}|^{\frac{1}{2}}<\frac{\varepsilon}{2}$. Thus, we have shown that the sequence $\left(V_g(Du_i)\right)_{i=1}^{\infty}$ is Cauchy in $L^1(Q')$ and therefore there exists a function $w\in L^1(Q')$ such that $V_g(Du_i)\to w$ in $L^1(Q')$ as $i\to\infty$. This implies that there exists a subsequence $(V_g(Du_{i_j}))_{j=1}^{\infty}$ converging to $w$ almost everywhere in $Q'$. Now the fact that $V_g$ has a continuous inverse yields 
\[
 Du_{i_j}=V_g^{-1}(V_g(Du_{i_j}))\to V_g^{-1}(w)=:v
\]
almost everywhere in $Q'$.

By Fatou's lemma we obtain 
\[
\int_{Q'}G(|v|)\dx\dt \leq\liminf_{j\to\infty}\int_{Q'}G\left(\left|Du_{i_j}\right|\right)\dx\dt\leq M_1,
\]
that is, $v\in L^G(Q')$. Now, since $u_i\to u$ almost everywhere in $Q'$, we have for any $\phi\in C^{\infty}_c(Q')$ that 
\[
\int_{Q'}uD\phi\dx\dt=\lim_{j\to\infty}\int_{Q'}u_{i_j}D\phi\dx\dt=-\lim_{j\to\infty}\int_{Q'}Du_{i_j}\,\phi\dx\dt
=-\int_{Q'}v\phi\dx\dt
\]
by Lebesgue's dominated convergence theorem and the definition of weak gradient, showing that $v=Du$. Thus, we have $Du_{i_j}\to Du$ 
almost everywhere in $Q'$ for any $Q'\Subset Q$, which implies that $Du_{i_j}\to Du$ almost everywhere in $Q$. 

To show that, in fact, the whole sequence $(Du_i)_{i=1}^{\infty}$ converges to $Du$ almost everywhere assume the contrary. Then there 
exists a subsequence $(Du_{i_k})_{k=1}^{\infty}$ such that for some $\varepsilon'>0$ we have $\left|Du_{i_k}-Du\right|\geq \varepsilon'$ for every $k$. However, the above reasoning holds if we replace $u_i$ by $u_{i_k}$ and thus, we find a subsubsequence 
$(u_{i_{k_j}})_{j=1}^{\infty}$ such that $|Du_{i_{k_j}}-Du|\to 0$ almost everywhere as $j\to \infty$. This is a contradiction, which proves that 
$Du_i\to Du$ almost everywhere in $Q$, and the proof is complete. 
\end{proof}

\begin{Theorem}\label{superlimit}
 Let $Q\subset\Omega_T, M\geq 1$, and let $(u_i)_{i=1}^{\infty}$ be a sequence of weak supersolutions in $Q$ such that $|u_i|\leq M$ for every $i\in\N$ and $u_i\to u$ almost everywhere in $Q$. Then $u$ is a weak supersolution in $Q$.
\end{Theorem}
\begin{proof}
Let $\eta\in C^{\infty}_c(Q)$ and choose $Q'\Subset Q$ such that $\supp\eta\subset Q'$. Since 
\[
 \begin{split}
  &\left|\int_{Q}\A(Du_i)\cdot D\eta\dx\dt-\int_{Q}\A(Du)\cdot D\eta\dx\dt\right|\\
  &\qquad\qquad\qquad\leq\int_{Q'}|\A(Du_i)-\A(Du)||D\eta|\dx\dt\\
  &\qquad\qquad\qquad\leq ||D\eta||_{L^{\infty}\left(Q'\right)}||\A(Du_i)-\A(Du)||_{L^1\left(Q'\right)}
 \end{split}
\]
and 
\[
 \begin{split}
  \left|\int_{Q}u_i\partial_t\eta\dx\dt-\int_{Q}u\partial_t\eta\dx\dt\right|
  &\leq\int_{Q'}|u_i-u||\partial_t\eta|\dx\dt\\
  &\leq ||\partial_t\eta||_{L^{\infty}\left(Q'\right)}||u_i-u||_{L^1\left(Q'\right)},
 \end{split}
\]
it suffices to show that $u_i\to u$ and $\A(Du_i)\to\A(Du)$ in $L^1(Q')$. The former follows by Lebesgue's dominated convergence theorem, since $u_i\to u$ almost everywhere in $Q$ and $|u_i-u|\leq 2M \in L^1(Q')$.

To show the latter, we observe that by Lemma~\ref{gradientconvergence} $Du_i\to Du$ almost everywhere in $Q'$. This implies that $\A(Du_i)\to\A(Du)$ almost everywhere in $Q'$ by the continuity of $\A$. A completely analogous application of the Caccioppoli estimate as in Lemma~\ref{gradientconvergence} gives a constant $M_2\geq 1$ independent of $i$ such that 
\[
\int_{Q'}\widetilde G(|\A(Du_i)|)\dx\dt\leq M_2.
\]
Therefore, by Fatou's lemma we have 
\[
 \int_{Q'}\widetilde G(|\A(Du)|)\dx\dt\leq\liminf_{i\to\infty}\int_{Q'}\widetilde G(|\A(Du_i)|)\dx\dt\leq M_2,
\]
which implies 
\[
 ||\A(Du)||_{L^{\widetilde G}(Q')},||\A(Du_i)||_{L^{\widetilde G}(Q')}\leq M_2.
\]

Denote $E_{\gamma}:=Q'\cap\{|\A(Du_i)-\A(Du)|\geq\gamma\}$, where $\gamma>0$ will be chosen shortly. By H\"older's inequality and Lemma~\ref{characteristicnorm} we obtain 
\[
 \begin{split}
  \int_{Q'}|\A(Du_i)-\A(Du)|\dx\dt &= \int_{Q'\setminus E_{\gamma}}|\A(Du_i)-\A(Du)|\dx\dt\\
  &\quad+\int_{E_{\gamma}}|\A(Du_i)-\A(Du)|\dx\dt\\
  &\leq\gamma|Q|+2||\chi_{E_{\gamma}}||_{L^G(Q')}||\A(Du_i)-\A(Du)||_{L^{\widetilde G}(Q')}\\
  &\leq\gamma|Q|+4M_2\max\left\{|E_{\gamma}|^{\frac{1}{g_1}},|E_{\gamma}|^{\frac{1}{g_0}}\right\}.
 \end{split}
\]
Now, for a fixed $\varepsilon>0$, we first choose $\gamma=\frac{\varepsilon}{2|Q|}$ and then $i$ large enough such that
\[
 4M_2\max\left\{|E_{\gamma}|^{\frac{1}{g_1}},|E_{\gamma}|^{\frac{1}{g_0}}\right\}<\frac{\varepsilon}{2}.
\]
This can be done, since in a set with finite measure convergence almost everywhere implies convergence in measure. Thus, we have shown 
that $\A(Du_i)\to\A(Du)$ in $L^1(Q')$ and the proof is complete. \qedhere
\end{proof}

\section{Qualitative properties of solutions}\label{qualitativeproperties}

In this section we prove that weak supersolutions always have a lower semicontinuous representative. For this we need boundedness of nonnegative weak subsolutions which is also an interesting result in its own right. For the evolutionary $p$-Laplace equation the lower semicontinuity of weak supersolutions was first proved in \cite{Kuusi}. We remark that the results of this section hold also for more general vector fields $\A(x,t,\xi)$ being measurable in $(x,t)$, continuous in $\xi$, and satisfying the weaker structural conditions \eqref{assumptions1}.

In order to choose the correct geometry we need to understand the scaling of the equation. Suppose $u$ is a weak solution in $B_\rho\times(-\theta,0)$. Then 
\[
 \bar u(x,t):=\frac{u\left(\rho x,\theta t\right)}{k}
\]
is a weak solution in $B_1\times(-1,0)$ with $\A$ replaced by 
\[
 \bar
 \A(\xi):=\frac{k}{\rho}G\left(\frac{k}{\rho}\right)^{-1}\A\left(\frac{k}{\rho}\xi\right)
\]
if and only if 
\begin{equation}\label{geometry}
 \theta=k^2G\left(\frac{k}{\rho}\right)^{-1}.
\end{equation}
Observe that $\bar\A$ satisfies the same structural conditions as $\A$ with $g$ replaced by 
\[
 \bar g(s):= \frac{k}{\rho}G\left(\frac{k}{\rho}\right)^{-1}g\left(\frac{k}{\rho}s\right),
\]
and furthermore, $\bar g$ satisfies the Orlicz condition \eqref{O.property} with the same constants as $g$. 

We begin by proving an {\em a priori} result using a standard De Giorgi iteration. 

\begin{Lemma}\label{k_lemma}
 Let $(x_0,t_0)\in\Omega_T$, $k>0$, and take $\rho,\theta>0$ such that 
\[
 Q(\rho,\theta):=B_\rho(x_0)\times(t_0-\theta,t_0)\Subset\Omega_T
\]
and \eqref{geometry} holds. If $u$ is a nonnegative weak subsolution in $\Omega_T$, then there exists a constant $\sigma\equiv\sigma(n,g_0,g_1,\nu,L)\in(0,1)$ such that whenever 
 \begin{equation}\label{k_condition}
  \mint_{Q(\rho,\theta)}\left(G\left(\frac{u}{\rho}\right)+\frac{u^2}{\theta}\right)\dx\dt\leq \sigma G\left(\frac{k}{\rho}\right)
 \end{equation}
holds, we have 
\[
 \esssup_{Q(\rho/2,\theta/2)}u\leq k.
\]
\end{Lemma}
\begin{proof}
Set for $j\in\N_0$ 
\[
 \rho_j=\left(1+2^{-j}\right)\frac\rho2,\qquad \theta_j=\left(1+2^{-j}\right)\frac\theta2,\qquad k_j=\left(1-2^{-j}\right)k
\]
and define 
\[
 Q_j:=Q(\rho_j,\theta_j). 
\]
Moreover, for technical reasons we introduce 
\[
 \tilde \rho_j = \frac{\rho_j+\rho_{j+1}}{2}, \qquad \tilde\theta_j=\frac{\theta_j+\theta_{j+1}}{2},\qquad \widetilde Q_j=Q(\tilde\rho_j,\tilde\theta_j).
\]
Observe that $Q_{j+1}\subset\widetilde Q_j\subset Q_j$. Let $\varphi_j\in C^\infty(Q_j), \tilde\varphi_j\in C^\infty(\widetilde Q_j)$ be such that $\varphi_j, \tilde\varphi_j$ vanish on $\partial_p Q_j,\partial_p\widetilde Q_j$, respectively, $0\leq\varphi_j,\tilde\varphi_j\leq 1$, $\varphi_j=1$ in $\widetilde Q_j$, $\tilde\varphi_j=1$ in $Q_{j+1}$, and 
\[
 |D\varphi_j|,|D\tilde\varphi_j|\leq c\frac{2^j}{\rho},\qquad |\partial_t\varphi_j|,|\partial_t\tilde\varphi_j|\leq c\frac{2^j}{\theta},
\]
where $c$ is a universal constant. 

Define 
\[
 Y_j:=G\left(\frac{k}{\rho}\right)^{-1}\mint_{Q_j}\left(G\left(\frac{(u-k_j)_+}{\rho}\right)+\frac{(u-k_j)_+^2}{\theta}\right)\dx\dt.
\]
The aim is to use De Giorgi's iteration method and for that we need to estimate $Y_{j+1}$. We shall estimate the two integral terms appearing in $Y_{j+1}$ separately. Since in the support of $(u-k_{j+1})_+$ we have 
\begin{equation}\label{support}
 (u-k_j)_+\geq k_{j+1}-k_j = 2^{-j-1}k,
\end{equation}
Lemma~\ref{parabolicsobolev} with $q=1$ yields 
\begin{align*}
 \mint_{Q_{j+1}}&G\left(\frac{(u-k_{j+1})_+}{\rho}\right)\dx\dt \notag \\ 
 &\leq c\,2^{2j/n}k^{-2/n}\mint_{\widetilde Q_j}G\left(\frac{(u-k_j)_+\tilde\varphi_j}{\tilde\rho_j}\right)\big((u-k_j)_+\tilde\varphi_j\big)^{2/n}\dx\dt\notag\\
 &\leq c\,2^{2j/n}k^{-2/n}\esssup_{(t_0-\tilde\theta_j,t_0)}\left(\mint_{B_{\tilde\rho_j}}(u-k_j)_+^2\tilde\varphi_j^2\dx\right)^{1/n}\notag\\
 &\qquad\qquad\qquad\qquad\qquad\qquad\qquad\quad\times\mint_{\widetilde Q_j}G\Big(\left|D\big((u-k_j)_+\tilde\varphi_j\big)\right|\Big)\dx\dt\notag\\
 &\leq c\,2^{2j/n}k^{-2/n}\theta^{1/n}\left(\frac{1}{\theta_j}\esssup_{(t_0-\theta_j,t_0)}\mint_{B_{\rho_j}}(u-k_j)_+^2\varphi_j^{g_1}\dx\right)^{1/n}\notag\\
 &\qquad\quad\times\left(\mint_{Q_j}G\big(|D(u-k_j)_+|\big)\varphi_j^{g_1}\dx\dt + \mint_{Q_j}G\big((u-k_j)_+|D\tilde\varphi_j|\big)\dx\dt\right).
\end{align*}
The Caccioppoli inequality, Lemma~\ref{caccioppoli}, gives 
\begin{align*}
 \frac{1}{\theta_j}\esssup_{(t_0-\theta_j,t_0)}\mint_{B_{\rho_j}}&(u-k_j)_+^2\varphi_j^{g_1}\dx + \mint_{Q_j}G(|D(u-k_j)_+|)\varphi_j^{g_1}\dx\dt\\
  &\leq c\mint_{Q_j}G\big(|D\varphi_j|(u-k_j)_+\big)\dx\dt + c\mint_{Q_j}(u-k_j)_+^2|\partial_t\varphi_j|\dx\dt,
\end{align*}
since $\varphi_j$ vanishes at $t=t_0-\theta_j$, and thus, using also \eqref{geometry}, we obtain 
\begin{align*}
 &\mint_{Q_{j+1}}G\left(\frac{(u-k_{j+1})_+}{\rho}\right)\dx\dt \notag \\
 &\leq c\,2^{2j/n}k^{-2/n}\theta^{1/n}\left(\mint_{Q_j}G\left(2^j\frac{(u-k_j)_+}{\rho}\right)\dx\dt+\mint_{Q_j}2^j\frac{(u-k_j)_+^2}{\theta}\dx\dt\right)^{1+1/n}\notag\\
 &\leq c\, 2^{(2/n+g_1(1+1/n))j}\left[\theta k^{-2} G\left(\frac{k}{\rho}\right)\right]^{1/n}G\left(\frac{k}{\rho}\right)Y_j^{1+1/n}
 = c\,b_1^jG\left(\frac{k}{\rho}\right)Y_j^{1+1/n},
\end{align*}
where $b_1=2^{2/n+g_1(1+1/n)}$. 

For the second term we apply Lemma~\ref{parabolicsobolev} with $q=2n/(n+2)$. The mapping $s\mapsto sG(s)^{-1/q}$ is decreasing due to the assumption $g_0>2n/(n+2)$, and therefore by \eqref{support} in the support of $(u-k_{j+1})_+$ we have 
\[
 \frac{(u-k_j)_+}{\rho} G\left(\frac{(u-k_j)_+}{\rho}\right)^{-1/q}\leq \frac{k}{2^{j+1}\rho} G\left(\frac{k}{2^{j+1}\rho}\right)^{-1/q},
\]
which implies 
\[
 (u-k_j)_+\leq 2^{(g_1/q-1)(j+1)}k\,G\left(\frac{k}{\rho}\right)^{-1/q}G\left(\frac{(u-k_j)_+}{\tilde\rho_j}\right)^{1/q}.
\]
Combining this with the Caccioppoli inequality as above yields 
\begin{align*}
 \mint_{Q_{j+1}}&\frac{(u-k_{j+1})_+^2}{\theta}\dx\dt\notag\\
 &\leq c\,2^{(g_1/q-1)j}\theta^{-1}k G\left(\frac{k}{\rho}\right)^{-1/q}\mint_{\widetilde Q_j}G\left(\frac{(u-k_j)_+\tilde\varphi_j}{\tilde\rho_j}\right)^{1/q}(u-k_j)_+\tilde\varphi_j\dx\dt\notag\\
 &\leq c\,2^{(g_1/q-1)j}\theta^{-1}k G\left(\frac{k}{\rho}\right)^{-1/q}\esssup_{(t_0-\tilde\theta_j,t_0)}\left(\mint_{B_{\tilde\rho_j}}(u-k_j)_+^2\tilde\varphi_j^2\dx\right)^{1/2}\notag\\
 &\qquad\qquad\qquad\qquad\qquad\qquad\qquad\quad\times\left(\mint_{\widetilde Q_j}G\Big(\left|D\big((u-k_j)_+\tilde\varphi_j\big)\right|\Big)\dx\dt\right)^{1/q}\notag\\
 &\leq c\,2^{(g_1/q-1)j}\theta^{-1/2}k G\left(\frac{k}{\rho}\right)^{-1/q}\left(\frac{1}{\theta_j}\esssup_{(t_0-\theta_j,t_0)}\mint_{B_{\rho_j}}(u-k_j)_+^2\varphi_j^{g_1}\dx\right)^{1/2}\notag\\
 &\qquad\times\left(\mint_{Q_j}G\big(|D(u-k_j)_+|\big)\varphi_j^{g_1}\dx\dt + \mint_{Q_j}G\big((u-k_j)_+|D\tilde\varphi_j|\big)\dx\dt\right)^{1/q}\notag\\
 &\leq c\,2^{((3/2+2/n)g_1-1)j}\left[\theta k^{-2}G\left(\frac{k}{\rho}\right)\right]^{-1/2}G\left(\frac{k}{\rho}\right)Y_j^{1+1/n}
 = c\,b_2^jG\left(\frac{k}{\rho}\right)Y_j^{1+1/n},
\end{align*}
where $b_2=2^{(3/2+2/n)g_1-1}$. 

By putting the two estimates together we obtain 
\begin{align*}
 Y_{j+1} &= G\left(\frac{k}{\rho}\right)^{-1}\left(\mint_{Q_{j+1}}G\left(\frac{(u-k_{j+1})_+}{\rho}\right)\dx\dt+\mint_{Q_{j+1}}\frac{(u-k_{j+1})_+^2}{\theta}\dx\dt\right)\notag\\
 &\leq c\,b^jY_j^{1+1/n},
\end{align*}
where $b=\max\{b_1,b_2\}$. Now a standard hyper-geometric iteration lemma, see Lemma~4.1 in \cite{DiBenedetto}, yields $Y_j\to 0$ as $j\to\infty$ provided that $Y_0\leq c^{-n}b^{-n^2}$. But this condition is precisely \eqref{k_condition} if we choose $\sigma=c^{-n}b^{-n^2}$. Therefore $u\leq k$ almost everywhere in $Q(\rho/2,\theta/2)$, as required. 
\end{proof}

In order to prove the boundedness of nonnegative weak subsolutions we still need to show that there always exists a number $k$ that satisfies \eqref{k_condition}. Due to the general nature of the equation this can only be done implicitly so that, at least with our method, it is not possible to obtain a nice {\em a priori} estimate, like in the case of the $p$-Laplacian (see \cite{DiBenedetto}). 

To this end, define 
\[
 a:= \liminf_{s\to\infty}\frac{s^2}{G(s)},\qquad A:= \limsup_{s\to\infty}\frac{s^2}{G(s)}.
\]
Depending on the growth of $G$ we consider three separate cases. The case $a=A=0$ is the \emph{degenerate} case. When $a=A=\infty$ we have the \emph{singular} case. The remaining case where either $a$ or $A$ is strictly positive and finite, or $a=0$ and $A=\infty$, we shall call the \emph{intermediate} case. Notice that when $g_0>2$ ($g_1<2$) we must be in the degenerate (singular) case, and on the other hand in the degenerate (singular) case we always have $g_1>2$ ($g_0<2$). 

\begin{Theorem}\label{boundednesstheorem}
 Let $u$ be a nonnegative weak subsolution in $\Omega_T$. Then $u\in L_{\textrm{loc}}^\infty(\Omega_T)$.
\end{Theorem}
\begin{proof}
The idea is to show that in each case there exists some finite $k$ and a neighborhood $Q(\rho,\theta)$ of $(x_0,t_0)\in\Omega_T$ such that \eqref{k_condition} holds. Then 
\[
\esssup_{Q(\rho/2,\theta/2)}u\leq k
\]
by Lemma~\ref{k_lemma} and the claim follows. Recall that the notion of weak subsolution includes that $u\in V_{\textrm{loc}}^{2,G}(\Omega_T)$, and hence in particular 
\[
\int_Q \big(G(u)+u^2\big)\dx\dt < \infty
\]
for any $Q\Subset\Omega_T$. 

{\bf Intermediate case.} If either $1 \leq a < \infty$ or $ 0 < A \leq 1$ there clearly exist a constant $1\leq M<\infty$ and a sequence $(s_m)_{m=0}^\infty$ such that $\lim_{m\to\infty}s_m=\infty$ and for every $m\in\N_0$ we have 
\begin{equation}\label{M}
 \frac{1}{M} \leq \frac{s_m^2}{G(s_m)} \leq M.
\end{equation}
On the other hand, if $a<1<A$ we can always find a sequence $(s_m)_{m=0}^\infty$ such that $\lim_{m\to\infty}s_m=\infty$ and $\frac{s_m^2}{G(s_m)}=1$ for every $m\in\N_0$, and thus in this case \eqref{M} holds with $M=1$. 

Fix a radius $0<r<1/M$ such that $Q(r,r)\Subset\Omega_T$ and let $m^*$ be the smallest $m$ that satisfies 
\[
s_{m} \geq \left(\frac{1}{\sigma r}\mint_{Q(r,r)}\left(G\left(\frac{u}{r}\right)+M\left(\frac{u}{r}\right)^2\right)\dx\dt\right)^{1/2}.
\]
Set $k=rs_{m^*}$ and choose $\rho=r$ and $\theta=k^2 G\left(\frac{k}{r}\right)^{-1}$. Observe that \eqref{geometry} holds and, moreover, 
\[
 \theta=\frac{k^2}{G\left(\frac{k}{r}\right)}=\frac{s_{m^*}^2}{G\left(s_{m^*}\right)}r^2\leq Mr^2 \leq r
\]
by \eqref{M}. Now 
\begin{align*}
 \mint_{Q(\rho,\theta)}&\left(G\left(\frac{u}{\rho}\right)+\frac{u^2}{\theta}\right)\dx\dt
 = \frac{G\left(\frac{k}{r}\right)}{|B_r|k^2}\int_{Q(\rho,\theta)}\left(G\left(\frac{u}{r}\right)+\frac{G(s_{m^*})}{s_{m^*}^2}\left(\frac{u}{r}\right)^2\right)\dx\dt\notag\\
 &\leq \frac{G\left(\frac{k}{r}\right)}{r s_{m^*}^2}\mint_{Q(r,r)}\left(G\left(\frac{u}{r}\right)+M\left(\frac{u}{r}\right)^2\right)\dx\dt 
 \leq \sigma G\left(\frac{k}{\rho}\right).
\end{align*}

{\bf Degenerate case.} Since now $\limsup_{s\to\infty}\frac{s^2}{G(s)}=0$, there exists $s_0\geq 1$ such that 
\begin{equation}\label{s_0}
 \sup_{s\geq s_0}\frac{s^2}{G(s)}\leq 1.
\end{equation}
Set for $m\in\N_0$ 
\[
 D_m:=\sup_{s\geq s_0+m}\frac{s^2}{G(s)}
\]
and observe that $D_m\leq D_0\leq 1$ by \eqref{s_0}. Using the very definition of supremum, for every $m\in\N_0$ there exists $s_m\geq s_0+m$ such that 
\begin{equation}\label{D_m}
 \frac{s_m^2}{G(s_m)}\geq \frac12 D_m.
\end{equation}

Let $r>0$ be such that $Q(r,r^2)\Subset\Omega_T$. We clearly have $\lim_{m\to\infty}s_m=\infty$ so that we may find the smallest $m$, which we again call $m^*$, such that 
\[
 s_m \geq \left(\frac{c^*}{\sigma}\mint_{Q(r,r^2)}G\left(\frac{u}{r}\right)\dx\dt\right)^{1/2}.
\]
The constant $c^*$ shall be determined shortly. Again, set $k=rs_{m^*}$ and choose $\rho=r$ and $\theta=k^2 G\left(\frac{k}{r}\right)^{-1}$. Since $\frac{s_{m^*}^2}{G(s_{m^*})}\leq D_m\leq 1$, we have 
\[
 \theta=\frac{k^2}{G\left(\frac{k}{r}\right)}=\frac{s_{m^*}^2}{G(s_{m^*})}r^2 \leq r^2.
\]
Take $0<\eps<1$ to be chosen later. In the set $\{u\geq \eps k\}$ we obtain by a change of variables and \eqref{D_m} that 
\begin{align*}
 \frac{u^2}{G\left(\frac{u}{r}\right)}&\leq\sup_{s\geq\eps k}\frac{s^2}{G\left(\frac{s}{r}\right)}=\eps^2r^2\sup_{s\geq s_{m^*}}\frac{s^2}{G(\eps s)}\notag\\
 &\leq \eps^{2-g_1}r^2D_{m^*}\leq 2\eps^{2-g_1}r^2\frac{s_{m^*}^2}{G(s_{m^*})}=2\eps^{2-g_1}\theta.
\end{align*}
Therefore, 
\begin{align*}
 \mint_{Q(\rho,\theta)}\left(G\left(\frac{u}{\rho}\right)+\frac{u^2}{\theta}\right)&\dx\dt
 =\mint_{Q(\rho,\theta)}\left(G\left(\frac{u}{r}\right)+\frac{u^2}{k^2}G\left(\frac{k}{r}\right)\right)\chi_{\{u<\eps k\}}\dx\dt\\
 &\qquad\qquad+\frac{G\left(\frac{k}{r}\right)}{|B_r|k^2}\int_{Q(\rho,\theta)}\left(G\left(\frac{u}{r}\right)+\frac{u^2}{\theta}\right)\chi_{\{u\geq\eps k\}}\dx\dt\\
 &\leq \left(\eps^{g_0}+\eps^2+\frac{1+2\eps^{2-g_1}}{s_{m^*}^2}\mint_{Q(r,r^2)}G\left(\frac{u}{r}\right)\dx\dt\right)G\left(\frac{k}{r}\right)\\
 &\leq \left(2\eps^{\min\{g_0,2\}}+\frac{1+2\eps^{2-g_1}}{c^*}\sigma\right)G\left(\frac{k}{r}\right)
 =\sigma G\left(\frac{k}{\rho}\right),
\end{align*}
if we choose $\eps=(\sigma/4)^{1/\min\{g_0,2\}}$ and $c^*=2\left(1+2\eps^{2-g_1}\right)$. 

{\bf Singular case.} The proof is very similar to the degenerate case and therefore some details shall be omitted. A change of variables gives 
\[
 \liminf_{s\to\infty}\frac{s^2}{G(s)}=\left(\limsup_{s\to\infty}\frac{s}{G^{-1}(s^2)}\right)^{-2},
\]
and thus the condition $a=\infty$ is equivalent to $\limsup_{s\to\infty}\frac{s}{G^{-1}(s^2)}=0$. Proceeding as in the degenerate case, choose $s_0\geq 1$ such that 
\[
 \sup_{s\geq s_0}\frac{s}{G^{-1}(s^2)}\leq 1,
\]
set for $m\in\N_0$ 
\[
 S_m:=\sup_{s\geq s_0+m}\frac{s}{G^{-1}(s^2)},
\]
and construct the sequence $(s_m)_{m=0}^\infty$ such that $s_m\geq s_0+m$ and 
\[
 \frac{s_m}{G^{-1}(s_m^2)}\geq\frac12 S_m.
\]

We again fix $r>0$ such that $Q(r,r^2)\Subset\Omega_T$ and this time take $m^*$ to be the smallest $m$ for which 
\[
 s_m\geq \left(\frac{c^*}{\sigma}\mint_{Q(r,r^2)}\frac{u^2}{r^2}\dx\dt\right)^{1/(n+2-2n/g_0)}
\]
holds for some $c^*$ to be fixed. Set $k=rs_{m^*}$ and choose $\displaystyle{\rho=k\left[G^{-1}\left(\frac{k^2}{r^2}\right)\right]^{-1}}$ and $\theta=r^2$. Notice that \eqref{geometry} holds and we again have $Q(\rho,\theta)\subset Q(r,r^2)$, since 
\[
\rho = \frac{k}{G^{-1}\left(\frac{k^2}{r^2}\right)}=\frac{s_{m^*}}{G^{-1}(s_{m^*}^2)}r\leq S_{m^*}r\leq r.
\]

Let $\eps$ be the same as above. A similar calculation as before shows that 
\[
 \frac{u}{G^{-1}\left(\frac{u^2}{r^2}\right)}\leq 2\eps^{1-2/g_0}\rho
\]
in the set $\{u\geq\eps k\}$. Thus 
\begin{align*}
\mint_{Q(\rho,\theta)}&\left(G\left(\frac{u}{\rho}\right)+\frac{u^2}{\theta}\right)\dx\dt
 =\mint_{Q(\rho,\theta)}\left(G\left(\frac{u}{k}G^{-1}\left(\frac{k^2}{r^2}\right)\right)+\frac{u^2}{r^2}\right)\chi_{\{u<\eps k\}}\dx\dt\\
 &\qquad\qquad\qquad\qquad\qquad+\frac{\left[G^{-1}(s_{m^*}^2)\right]^n}{|B_r|r^2s_{m^*}^n}\int_{Q(\rho,\theta)}\left(G\left(\frac{u}{\rho}\right)+\frac{u^2}{r^2}\right)\chi_{\{u\geq\eps k\}}\dx\dt\\
 &\qquad\leq \left(\eps^{g_0}+\eps^2+s_{m^*}^{2n/g_0-(n+2)}\left(\left(2\eps^{1-2/g_0}\right)^{g_1}+1\right)\mint_{Q(r,r^2)}\frac{u^2}{r^2}\dx\dt\right)\frac{k^2}{r^2}\\
 &\qquad\leq \left(2\eps^{\min\{g_0,2\}}+\frac{\left(2\eps^{1-2/g_0}\right)^{g_1}+1}{c^*}\sigma\right)\frac{k^2}{r^2}
 =\sigma G\left(\frac{k}{\rho}\right)
\end{align*}
upon choosing $c^*=2\left(\left(2\eps^{1-2/g_0}\right)^{g_1}+1\right)$. 

We have shown that in all three cases a finite $k$ exists, which proves that nonnegative weak subsolutions are locally bounded. 
\end{proof}

\begin{Remark}
Observe that we get no quantitative information about the size of $k$ in any of the cases. This is due to the fact that we only have qualitative information about the sequence $s_m$, so that $s_{m^*}$ could be arbitrarily large, although finite. 
\end{Remark}

If $u$ is only a weak subsolution but not necessarily nonnegative, we may apply the result to $\max\{u,0\}$ which is a nonnegative weak subsolution by Lemma~\ref{minsuper}. Similarly, if $u$ is a weak supersolution, then $\max\{-u,0\}$ is a nonnegative weak subsolution. Hence we obtain the following corollary.

\begin{Corollary}\label{supercorollary}
 Let $u$ be a weak supersolution (subsolution) in $\Omega_T$. Then $u$ is locally essentially bounded from below (above). In particular, if $u$ is a weak solution in $\Omega_T$, then $u\in L_{\textrm{loc}}^\infty(\Omega_T)$.
\end{Corollary}

After the boundedness of nonnegative weak subsolutions has been established, we obtain the lower semicontinuity of supersolutions fairly easily by using the {\em a priori} estimate in \eqref{k_condition} and Lebesgue's differentiation theorem. For this we define the $\essliminf$-regularization of a function $u$ that is bounded from below as 
\begin{equation}\label{essliminfreg}
\hat u(x,t):=\lim_{r\to 0}\essinf_{Q_r(x,t)\cap\Omega_T}u,
\end{equation}
where $Q_r(x,t):=B_r(x)\times\left(t-\frac12 r^2,t+\frac12 r^2\right)$. First we need a simple lemma. 

\begin{Lemma}\label{essliminflsc}
 Let $u$ be essentially bounded from below. Then $\hat u$ is lower semicontinuous.
\end{Lemma}
\begin{proof}
Fix $(x,t)\in\Omega_T$ and $\eps>0$. There exists $\rho_0>0$ such that $Q_{\rho_0}(x,t)\subset\Omega_T$ and 
\[
 \Big|\hat u(x,t)-\essinf_{Q_\rho(x,t)}u\Big|<\eps
\]
for every $0<\rho\leq \rho_0$. Fix such a $\rho$ and let $(y,s)\in Q_\rho(x,t)$. Observe that for all small enough $r>0$ we have $Q_r(y,s)\subset Q_\rho(x,t)$ and thus 
\[
 \hat u(y,s) = \lim_{r\to 0}\essinf_{Q_r(y,s)}u\geq \essinf_{Q_{\rho}(x,t)}u>\hat u(x,t)-\eps.
\]
Now 
\[
 \liminf_{(y,s)\to(x,t)}\hat u(y,s) = \lim_{\rho\to 0}\inf_{Q_\rho(x,t)}\hat u\geq \hat u(x,t)-\eps
\]
and the result follows by taking $\eps\to 0$.
\end{proof}

\begin{Theorem}\label{superlsc}
 Let $u$ be a weak supersolution in $\Omega_T$. Then $u=\hat u$ almost everywhere in $\Omega_T$ and, in particular, $u$ is lower semicontinuous after a redefinition in a set of measure zero.
\end{Theorem}
\begin{proof}
Since $u$ is bounded from below by Corollary~\ref{supercorollary}, $\hat u$ is lower semicontinuous by Lemma~\ref{essliminflsc}. Thus, it suffices to show that $u=\hat u$ almost everywhere in $\Omega_T$. 

Assume without loss of generality that $u\in L_{\textrm{loc}}^\infty(\Omega_T)$. Indeed, by Lemma~\ref{minsuper} the function $u_m:=\min\{u,m\}$ is a weak supersolution for every $m\in\N$ and, furthermore, $u_m\in L_{\textrm{loc}}^\infty(\Omega_T)$ by Corollary~\ref{supercorollary}. Therefore, if we show that $u_m=\hat u_m$ almost everywhere in $\Omega_T$, then by the inclusion 
\begin{multline*}
 \big\{(x,t)\in\Omega_T:u(x,t)\neq\hat u(x,t)\big\}\\
 \subset\big\{(x,t)\in\Omega_T:|u(x,t)|=\infty\big\}\cup \bigcup_{m=1}^\infty\big\{(x,t)\in\Omega_T:u_m(x,t)\neq\hat u_m(x,t)\big\}
\end{multline*}
and the fact that as an integrable function $u$ is finite almost everywhere in $\Omega_T$ it follows that $u=\hat u$ almost everywhere in $\Omega_T$.

Set $U:=\{(x,t)\in\Omega_T:|u(x,t)|<\infty\}$ and denote the set of Lebesgue points of $u$ in $\Omega_T$ by $V$. Since almost every point is a Lebesgue point and $u$ is integrable, we clearly have $|\Omega_T\setminus(U\cap V)|=0$. Hence, by showing that $u=\hat u$ in $U\cap V$ we obtain the result. 

To this end, fix $\eps>0$ and take $0<k<\eps$. Let $(x_0,t_0)\in U\cap V$ and denote $\widetilde Q(\rho,\theta):=B_\rho(x_0) \times\left(t_0-\frac12\theta,t_0+\frac12\theta\right)$, where $\rho$ and $\theta$ are chosen such that \eqref{k_condition} holds. Observe that Lemma~\ref{k_lemma} clearly holds also for cylinders of the type $\widetilde Q(\rho,\theta)\Subset\Omega_T$. We define the nonnegative weak subsolution 
\[
v:=\big(u(x_0,t_0)-u\big)_+
\]
and aim to show that 
\[
 \mint_{\widetilde Q(\rho,\theta)}\left(G\left(\frac{v}{\rho}\right)+\frac{v^2}{\theta}\right)\dx\dt\leq \sigma G\left(\frac{k}{\rho}\right)
\]
for $\rho$ small enough. Take $\delta>0$ to be fixed shortly and let $\mathcal K\Subset\Omega_T$ be a set including $(x_0,t_0)$. Since $v$ is locally bounded by Theorem~\ref{boundednesstheorem}, we may take $N$ to be the smallest positive integer such that $||v||_{L^\infty(\mathcal K)}< 2^N\delta$. Now for every $\rho>0$ satisfying $\widetilde Q(\rho,\theta)\subset \mathcal K$ we have 
\begin{align*}
 \mint_{\widetilde Q(\rho,\theta)}G\left(\frac{v}{\rho}\right)\chi_{\{v\geq\delta\}}\dx\dt&=\sum_{j=0}^{N-1}\mint_{\widetilde Q(\rho,\theta)}G\left(\frac{v}{\rho}\right)\chi_{\{2^j\delta\leq v<2^{j+1}\delta\}}\dx\dt\\
 &\leq \sum_{j=0}^{N-1}G\left(\frac{2^{j+1}\delta}{\rho}\right)\mint_{\widetilde Q(\rho,\theta)}\chi_{\{2^j\delta\leq v<2^{j+1}\delta\}}\dx\dt\\
 &\leq \sum_{j=0}^{N-1} G\left(\frac{2^{j+1}\delta}{\rho}\bigg(\mint_{\widetilde Q(\rho,\theta)}\chi_{\{2^j\delta\leq v<2^{j+1}\delta\}}\dx\dt\bigg)^{1/g_1}\right)\\
 &\leq \sum_{j=0}^{N-1} G\left(\frac2\rho\,\bigg(\mint_{\widetilde Q(\rho,\theta)}v^{g_1}\dx\dt\bigg)^{1/g_1}\right)\\
 &\leq G\left(\frac{2N}\rho\,\bigg(\mint_{\widetilde Q(\rho,\theta)}\big|u(x_0,t_0)-u\big|^{g_1}\dx\dt\bigg)^{1/g_1}\right).
\end{align*}

Since $(x_0,t_0)$ is a Lebesgue point and $u$ belongs to $L^p(\mathcal K)$ for every $1\leq p\leq\infty$, Lebesgue's differentiation theorem gives 
\[
 \lim_{\rho\to 0}\mint_{\widetilde Q(\rho,\theta)}\big|u(x_0,t_0)-u\big|^{p}\dx\dt = 0
\]
for $1\leq p<\infty$. In particular, we use this for $p=g_1$ and $p=2$ to find $\rho_0\equiv\rho_0(g_1,\sigma,k,N)$ and $\theta_0=k^2G\left(\frac{k}{\rho_0}\right)^{-1}$ such that $\widetilde Q(\rho_0,\theta_0)\subset \mathcal K$, 
\[
 \mint_{\widetilde Q(\rho_0,\theta_0)}\big|u(x_0,t_0)-u\big|^{g_1}\dx\dt \leq \left(\frac{\sigma}{8N}k\right)^{g_1},
\]
and
\[
 \mint_{\widetilde Q(\rho_0,\theta_0)}\big|u(x_0,t_0)-u\big|^{2}\dx\dt \leq \frac{\sigma}{2}k^2.
\]
Thus, by choosing $\delta=\frac{\sigma}{4}k$ we obtain the desired inequality 
\begin{align*}
 \mint_{\widetilde Q(\rho_0,\theta_0)}\left(G\left(\frac{v}{\rho_0}\right)+\frac{v^2}{\theta_0}\right)\dx\dt&\leq G\left(\frac{\delta}{\rho_0}\right)+\mint_{\widetilde Q(\rho_0,\theta_0)}G\left(\frac{v}{\rho_0}\right)\chi_{\{v\geq\delta\}}\dx\dt\\
 &\quad+\frac{1}{k^2}G\left(\frac{k}{\rho_0}\right)\mint_{\widetilde Q(\rho_0,\theta_0)}\big|u(x_0,t_0)-u\big|^{2}\dx\dt\\
 &\leq 2\,G\left(\frac{\sigma}{4}\frac{k}{\rho_0}\right)+\frac{\sigma}{2}G\left(\frac{k}{\rho_0}\right)\\
 &\leq \sigma G\left(\frac{k}{\rho_0}\right).
\end{align*}

Let $r_0>0$ be so small that $Q_{r_0}(x_0,t_0)\subset \widetilde Q(\rho_0/2,\theta_0/2)$. Then for every $0<r\leq r_0$ we have 
\[
 u(x_0,t_0)-\essinf_{Q_r(x_0,t_0)}u\leq \esssup_{Q_r(x_0,t_0)}v\leq \esssup_{\widetilde Q(\rho_0/2,\theta_0/2)}v \leq k < \eps 
\]
by Lemma~\ref{k_lemma}, and therefore 
\[
 u(x_0,t_0) < \hat u(x_0,t_0) + \eps.
\]
Since $\eps>0$ was arbitrary, we obtain 
\[
 u(x_0,t_0)\leq \hat u(x_0,t_0).
\]
The other direction follows from Lebesgue's differentiation theorem, since 
\[
 u(x_0,t_0) = \lim_{r\to 0}\mint_{Q_r(x_0,t_0)}u\dx\dt \geq \lim_{r\to 0}\essinf_{Q_r(x_0,t_0)}u = \hat u(x_0,t_0),
\]
and we are done. 
\end{proof}

\section{Obstacle problem}\label{obstacleproblem}

In this section we prove the existence of a unique solution to the bounded obstacle problem related to equation \eqref{eq}. Moreover, we show that if the obstacle is continuous, the same property is inherited by the solution. 

\begin{Definition}
 A function $u$ is a solution to the obstacle problem with the obstacle $\psi$ in $\Omega_T$, if $u$ is the smallest $\essliminf$-regularized (see \eqref{essliminfreg}) weak supersolution in $\Omega_T$ that lies above $\psi$ almost everywhere in $\Omega_T$.
\end{Definition}

Let us first consider merely bounded obstacles. The existence of a solution to the obstacle problem follows fairly easily using results from the previous sections. The idea of the proof is the same as in \cite{LP} for the $p$-Laplacian. 

\begin{Theorem}\label{bddexistence}
 Let $\psi\in L^\infty(\Omega_T)$. Then there exists a unique solution to the obstacle problem with the obstacle $\psi$ and, moreover, it belongs to $L^\infty(\Omega_T)$.
\end{Theorem}
\begin{proof}
By Theorem~\ref{superlsc} every weak supersolution $v$ has a representative such that $v=\hat v$ everywhere in $\Omega_T$. We consider the class of all such functions that lie above $\psi$ almost everywhere and show that the $\essliminf$-regularization of the pointwise infimum taken over this class meets the requirements of a solution.

To this end, denote the set of all weak supersolutions in $\Omega_T$ by $\mathcal{S}$ and define 
\[
 \mathcal{S}_\psi := \{v\in\mathcal{S}: v=\hat v\,\, \textrm{everywhere in} \,\Omega_T, v\geq\psi\,\, \textrm{almost everywhere in} \,\Omega_T\}.
\]
Since $\mathcal{S}_\psi$ includes the constant function $v\equiv M:=||\psi||_{L^\infty(\Omega_T)}$, it is nonempty. We set for all $(x,t)\in\Omega_T$ 
\[
 w(x,t):=\inf_{v\in\mathcal{S}_\psi}v(x,t)
\]
and aim to show that $\hat w$ is a solution. If $v$ is an $\essliminf$-regularized weak supersolution with $v\geq\psi$ almost everywhere in $\Omega_T$, then obviously $v\geq \hat w$ in $\Omega_T$. Thus, to prove that $\hat w$ is a solution we need to show that $\hat w\in\mathcal{S}_\psi$. In fact, it suffices to show that $w\in\mathcal{S}$, since then by Theorem~\ref{superlsc} $w=\hat w$ almost everywhere in $\Omega_T$, which implies $\hat w\in\mathcal{S}_\psi$. Notice also that $\hat w\in L^\infty(\Omega_T)$, since $w\geq\psi\geq -M$ almost everywhere and $w\leq M$ everywhere in $\Omega_T$, implying $|\hat w|\leq M$ in $\Omega_T$. 

If $v_1,v_2\in\mathcal{S}_\psi$, then by Lemma~\ref{minsuper} also $\min\{v_1,v_2\}\in\mathcal{S}_\psi$. Therefore, by Choquet's topological lemma, see p.~158 in \cite{HKM}, there exist a function $u$ and a decreasing sequence $(u_i)_{i=1}^\infty$ such that $u_i\in \mathcal{S}_\psi$ for every $i\in\N$, $u_i\to u$ everywhere as $i\to\infty$, and 
\[
 \liminf_{(y,s)\to(x,t)}u(y,s)=\liminf_{(y,s)\to(x,t)}w(y,s)
\]
for every $(x,t)\in\Omega_T$. Clearly $u\geq w$ in $\Omega_T$. Without loss of generality we may assume that $|u_i|\leq M$ in $\Omega_T$ for every $i\in\N$, and thus by Lemma~\ref{superlimit} $u$ is a weak supersolution in $\Omega_T$. But now at almost every $(x,t)\in\Omega_T$ we know that $u$ is lower semicontinuous by Theorem~\ref{superlsc}, which leads to 
\[
 w(x,t)\leq u(x,t)\leq \liminf_{(y,s)\to(x,t)}u(y,s)=\liminf_{(y,s)\to(x,t)}w(y,s)\leq w(x,t).
\]
Therefore, $w=u$ almost everywhere in $\Omega_T$, whence $w\in\mathcal{S}$ and we deduce that $\hat w$ is a solution to the obstacle problem with the obstacle $\psi$. Uniqueness is trivial, since $\hat w$ is the smallest function in $\mathcal{S}_\psi$. 
\end{proof}

For the rest of the section we shall turn our attention to continuous obstacles. Since $C^0(\overline\Omega_T^p)\subset L^\infty(\Omega_T)$, the existence of a unique solution is given by Theorem~\ref{bddexistence}. Now the interesting question is if the solution is also continuous. To answer this question we construct a sequence of functions using a modification of the Schwarz alternating method and show that the limit is indeed a continuous solution to the obstacle problem. By the uniqueness we then deduce that this solution must be the same as the one given by Theorem~\ref{bddexistence}. Moreover, we prove that whenever the solution does not coincide with the obstacle, it is in fact a weak solution. The proof follows the same guidelines as \cite{KKS} for parabolic $p$-Laplace type equations. 

Observe that when $\psi\in C^0(\overline\Omega_T^p)$, the solution to the obstacle problem in fact lies above $\psi$ everywhere.

We collect here two important results that will be needed later. They are proved in \cite{BL}. 

\begin{Theorem}\label{existenceofsolution}
 Let $Q:=B\times\Gamma\subset\Omega_T$, where $B$ is a ball in $\R^n$, and let $\theta\in C^0(\overline\Omega_T^p)$. Then there exists a unique weak solution $u$ in $Q$ such 
 that $u\in C^0(\overline Q^p)$ and $u=\theta$ on $\partial_p Q$.
\end{Theorem}

\begin{Theorem}\label{aprioritheorem}
 Let $u$ be the weak solution in $Q$ given by Theorem~\ref{existenceofsolution}. Then there exists a constant $c\equiv c(n,g_0,g_1,\nu,L)$ such that 
 \[
 ||Du||_{L^\infty(Q_R)}\leq c\left(\mint_{Q_{2R}}\big[G(|Du|)+1\big]\dx\dt\right)^{\max\left\{\frac12,\frac2{(n+2)g_0-2n}\right\}}
 \]
 for every parabolic cylinder $Q_{2R}\Subset Q$.
 \end{Theorem}

Let us begin with the construction of a candidate for a solution. 

\begin{Construction}\label{construction}
 Let $\mathcal{F}$ be a countable and dense family of cylinders defined as 
 \[
 \mathcal{F}=\big\{Q^k\subset\Omega_T:Q^k=B_{r_k}(x_k)\times (\tau_k,T),r_k,\tau_k\in\Q,x_k\in\Q^n\big\}.
 \]
Construct a sequence of functions $(\varphi_k)_{k=0}^{\infty}$ 
as follows: 
\[
 \varphi_0=\psi,\qquad\varphi_{k+1}=\max\{\varphi_k,v_k\},
\]
where $v_k$ is a weak solution in $Q^k$ with boundary values $\varphi_k$ on $\partial_p Q^k$ and $v_k=\varphi_k$ in 
$\Omega_T\setminus Q^k$. Denote the limit, if it exists, by 
\[
 u^* := \lim_{k\to\infty}\varphi_k.
\]
\end{Construction}

We easily deduce the following basic properties.

\begin{Proposition}\label{properties}
 \begin{enumerate}[(i)]
  \item We have $\varphi_k\geq \psi$ in $\Omega_T$ for every $k\in\N_0$. 
  \item The function $\varphi_k$ is continuous for every $k\in\N_0$. 
  \item We have 
  \[
   |\varphi_k|\leq\sup_{\Omega_T}|\psi|
  \]
  in $\Omega_T$ for every $k\in\N_0$. 
  \item The limit $u^*$ always exists and $u^*\geq\psi$ in $\Omega_T$. 
  \item If $v$ is an $\essliminf$-reqularized weak supersolution with $v\geq\psi$ almost everywhere in $\Omega_T$, then $v\geq u^*$ in $\Omega_T$.  
  \item The limit $u^*$ is lower semicontinuous.
 \end{enumerate}
\end{Proposition}
\begin{proof}
 \begin{enumerate}[(i)]
  \item By definition $\varphi_0=\psi$ and for every $k\in\N$ we have 
  \[
   \varphi_{k}=\max\{\varphi_{k-1},v_{k-1}\}\geq\varphi_{k-1}\geq\ldots\geq\varphi_0=\psi
  \]
  in $\Omega_T$. Note also that the sequence $(\varphi_k)_{k=0}^{\infty}$ is nondecreasing. 
  \item The function $\varphi_0=\psi$ is continuous by assumption. Now, if $\varphi_k$ is continuous for some $k\in\N_0$, then so is 
  $\varphi_{k+1}$ as the maximum of continuous functions, since $v_k$ is a weak solution in $Q^k$ and therefore continuous by 
  Theorem~\ref{existenceofsolution}. 
  \item Clearly $\varphi_0=\psi\leq\sup_{\Omega_T}|\psi|$. Assume then that the claim holds for some $k\in\N_0$. By the maximum principle, 
  Corollary~\ref{maximumprinciple}, we have 
  \[
  |v_k|\leq\sup_{Q^k}|v_k|\leq\sup_{\partial_p Q^k}|v_k|=\sup_{\partial_p Q^k}|\varphi_k|\leq\sup_{\Omega_T}|\psi|
  \]
  in $Q^k$, and thus,
  \[
   |\varphi_{k+1}|=\left\{ \begin{array}{ll}
    |v_k|, & v_k>\varphi_k\\[2mm]
    |\varphi_k|, & v_k\leq\varphi_k\\[1mm]
    \end{array}\right.
    \leq\sup_{\Omega_T}|\psi|
  \]
  in $\Omega_T$. 
  \item The sequence $(\varphi_k)_{k=0}^{\infty}$ is nondecreasing and uniformly bounded, thus the limit $u^*$ exists. Since all the 
  members of the sequence are above $\psi$ by $(i)$, also the limit has to be. 
  \item Suppose $v$ is an $\essliminf$-reqularized weak supersolution with $v\geq\psi$ almost everywhere in $\Omega_T$. We show that $v\geq\varphi_k$ everywhere in $\Omega_T$ for every $k\in\N_0$, which implies $v\geq u^*$ in $\Omega_T$. Since 
  \[
  v(x,t)=\lim_{r\to 0}\essinf_{Q_r(x,t)}v\geq\lim_{r\to 0}\essinf_{Q_r(x,t)}\psi=\psi(x,t)
  \]
  at every $(x,t)\in\Omega_T$, the assertion holds for $\varphi_0=\psi$. If it holds for some $k\in\N_0$, then by the comparison principle, Lemma~\ref{comparisonprinciple}, $v\geq v_k$ in $Q^k$, since $v\geq \varphi_k=v_k$ on $\partial_p Q^k$. Therefore, we also have $v\geq\varphi_{k+1}$ in $\Omega_T$. 
  \item The function $u^*$ is the limit of a nondecreasing sequence of continuous functions, hence it is lower semicontinuous. \qedhere
 \end{enumerate}
 \end{proof}

So-called $\A$-superharmonic functions are often studied in the theory of elliptic equations. As shown in \cite{HKM}, there is a strong connection between $\mathcal{A}$-superharmonic functions and weak supersolutions. The same idea can be used also in the parabolic setting. We shall call the corresponding functions $\A$-superparabolic. 
 
\begin{Definition}
 A function $u:\Omega_T\to (-\infty,\infty]$ is called \emph{$\A$-superparabolic}, if 
 \begin{enumerate}[(i)]
  \item $u$ is lower semicontinuous,
  \item $u$ is finite in a dense subset of $\Omega_T$,
  \item $u$ satisfies the comparison principle in every cylinder $Q\Subset\Omega_T$, that is, if $h\in C^0(\overline Q^p)$ is a 
  weak solution in $Q$ and $h\leq u$ on $\partial_pQ$, then $h\leq u$ in $Q$.
 \end{enumerate}
\end{Definition}
 
In order to prove that the limit $u^*$ of Construction~\ref{construction} is a solution to the obstacle problem, by Proposition~\ref{properties} it suffices to show that it is a weak supersolution. For this we prove that it is both $\A$-superparabolic and continuous. Let us first show the former. 
 
\begin{Lemma}\label{A-super}
 The limit $u^*$ of Construction~\ref{construction} is $\A$-superparabolic.
\end{Lemma}
\begin{proof}
By Proposition~\ref{properties} $u^*$ is lower semicontinuous and 
\[
 |u^*|=\lim_{k\to\infty}|\varphi_k|\leq\sup_{\Omega_T}|\psi|
\]
in $\Omega_T$. Thus, it is enough to show that $u^*$ satisfies the comparison principle in every cylinder $Q\Subset\Omega_T$. 

To this end, fix a cylinder $Q=K\times(t_1,t_2)\Subset\Omega_T$ and let $h\in C^0(\overline Q^p)$ be 
a weak solution in $Q$ such that $h\leq u^*$ on $\partial_pQ$. If we can show that $h\leq u^*$ in $Q$, we are done. Fix $\varepsilon>0$ 
and set for each $k\in\N$ 
\[
 E_k := \overline Q^p\cap\{\varphi_k>h-\varepsilon\}.
\]
By the continuity of $\varphi_k$ and $h$ the sets $E_k$ are open with respect to the relative topology. Since  $u^*=\lim_{k\to\infty}\varphi_k$, for any point $z=(x,t)\in\partial_pQ$ we find an integer $k_z\geq 1$ such that 
\[
\varphi_{k_z}(z)>u^*(z)-\varepsilon\geq h(z)-\varepsilon,
\]
implying that $z\in E_{k_z}$. Therefore, the sets $E_k$ form an open cover for $\partial_pQ$, and since $\partial_pQ$ is compact and the 
sequence $(\varphi_k)_{k=0}^{\infty}$ nondecreasing, there exists an integer $k_0\geq 1$ such that $\partial_pQ\subset E_{k_0}$. This, 
together with the fact that the set $E_{k_0}$ is open, implies that there exists $k_1\geq k_0$ such that the cylinder 
$Q^{k_1}\in\mathcal{F}$ satisfies 
\[
 \partial_pQ^{k_1}\cap\{t<t_2\}\subset E_{k_0}\qquad\textrm{and}\qquad Q\setminus E_{k_0}\subset Q^{k_1}\cap\{t<t_2\}.
\]

Now, since $v_{k_1}=\varphi_{k_1}$ on $\partial_pQ^{k_1}$, we have 
\[
 h\leq\varphi_{k_0}+\varepsilon\leq\varphi_{k_1}+\varepsilon=v_{k_1}+\varepsilon
\]
on $\partial_pQ^{k_1}\cap\{t<t_2\}$. Moreover, both $h$ and $v_{k_1}+\varepsilon$ are weak solutions in $Q^{k_1}\cap\{t<t_2\}$, and 
therefore, 
\[
 h\leq v_{k_1}+\varepsilon\leq\varphi_{k_1+1}+\varepsilon\leq u^*+\varepsilon
\]
in $Q^{k_1}\cap\{t<t_2\}$ by the comparison principle, Lemma~\ref{comparisonprinciple}. Hence, $h\leq u^*+\varepsilon$ also in 
$Q\setminus E_{k_0}$, and since 
\[
 h\leq\varphi_{k_0}+\varepsilon\leq u^*+\varepsilon
\]
trivially in $E_{k_0}$, we obtain the result by letting $\varepsilon$ tend to zero. 
\end{proof}

The next lemma shows that Construction~\ref{construction} is stable. 

\begin{Lemma}\label{order}
 The limit $u^*$ of Construction~\ref{construction} is independent of the order of the cylinders $Q^k$. 
\end{Lemma}
\begin{proof}
Construct functions $\widetilde\varphi_k$ and $\widetilde v_k$ as in Construction~\ref{construction} with the cylinders $Q^k$ taken 
in a different order than in the construction of $u^*$. Denote $\widetilde u^*:=\lim_{k\to\infty}\widetilde\varphi_k$. Clearly we have 
$\widetilde u^*\geq \varphi_0=\psi$ in $\Omega_T$. Assume then that $\widetilde u^*\geq \varphi_k$ in $\Omega_T$ for some $k\in\N_0$. Since 
$v_k$ is a weak solution in $Q^k$ with $v_k=\varphi_k$ on $\partial_pQ^k$ and $\widetilde u^*$ is $\A$-superparabolic by 
Lemma~\ref{A-super}, we have $v_k\leq \widetilde u^*$ in $Q^k$. Thus, $\varphi_{k+1}=\max\{\varphi_k,v_k\}\leq \widetilde u^*$ in $\Omega_T$, and by 
induction we obtain $u^*\leq\widetilde u^*$ in $\Omega_T$. Interchanging the roles of $u^*$ and $\widetilde u^*$ completes the proof. 
\end{proof}

This leads to the following comparison result. 

\begin{Lemma}\label{comparisonoflimits}
 Let $\psi_1$ and $\psi_2$ be continuous obstacles such that $\psi_1\leq\psi_2$ in $\Omega_T$. Then the corresponding limits $u^*_1$ and 
 $u^*_2$ of Construction~\ref{construction} satisfy $u^*_1\leq u^*_2$ in $\Omega_T$.
\end{Lemma}
\begin{proof}
By Lemma~\ref{order} we may take the cylinders $Q^k$ in the same order in the construction of both $u_1^*$ and $u_2^*$. Let 
$\varphi_k^i$ and $v_k^i$, $i\in\{1,2\}$, $k\in\N_0$, generate $u^*_1$ and $u^*_2$. By assumption $\varphi_0^1\leq\varphi_0^2$ in 
$\Omega_T$. Suppose then that $\varphi_k^1\leq\varphi_k^2$ in $\Omega_T$ for some $k\in\N_0$. Then we have 
\[
 v_k^1=\varphi_k^1\leq\varphi_k^2=v_k^2
\]
on $\partial_pQ^k$, and thus the comparison principle, Lemma~\ref{comparisonprinciple}, yields $v_k^1\leq v_k^2$ in $Q^k$. This implies 
$\varphi_{k+1}^1\leq\varphi_{k+1}^2$ in $\Omega_T$, and hence, $\varphi_k^1\leq\varphi_k^2$ in $\Omega_T$ for every $k\in\N_0$. 
Taking the limit $k\to\infty$ on both sides completes the proof. 
\end{proof}

We now have the necessary tools to prove the continuity of our candidate. 

\begin{Lemma}\label{continuity}
 The limit $u^*$ of Construction~\ref{construction} is continuous.
\end{Lemma}
\begin{proof}
Let $\varepsilon>0$ and $(x_0,t_0)\in\Omega_T$ be fixed. Since $\Omega_T$ is open and the obstacle $\psi$ continuous, there exists $r>0$ such that $Q_r:=B_r(x_0)\times\left(t_0-\frac12r^2,t_0+\frac12r^2\right)\Subset\Omega_T$ and 
\[
 \osc_{\overline Q_r^p}\psi\leq\frac{\varepsilon}{4}.
\]
Construct the modified obstacle 
\[
 \widetilde\psi := \left\{ \begin{array}{lll}
h & \textrm{in}\,\,Q_r\\
\psi & \textrm{in}\,\,\Omega_T\setminus Q_r
\end{array}\right.,
\]
where $h$ is a weak solution in $Q_r$ with $h=\psi$ on $\partial_pQ_r$. By Theorem~\ref{existenceofsolution} $h\in C^0(\overline Q_r^p)$, and 
thus, also $\widetilde\psi$ is continuous. Moreover, by the maximum principle, Corollary~\ref{maximumprinciple}, we have 
\[
 h-\psi\leq\sup_{\overline Q_r^p}h-\psi\leq\sup_{\partial_pQ_r}\psi-\inf_{\overline Q_r^p}\psi\leq\osc_{\overline Q_r^p}\psi
\]
in $\overline Q_r^p$. Similarly, $\psi-h\leq\osc_{\overline Q_r^p}\psi$ in $\overline Q_r^p$, and thus, we obtain 
\[
 |\psi-\widetilde\psi|\leq|\psi-h|\leq\osc_{\overline Q_r^p}\psi\leq\frac{\varepsilon}{4}
\]
in $\Omega_T$. 

Let $\widetilde u^*$ be the limit of Construction~\ref{construction} with the obstacle $\widetilde\psi$ generated by the functions 
$\widetilde\varphi_k$ and $\widetilde v_k$, $k\in\N_0$. Evidently adding a constant to the obstacle changes the corresponding 
limit by the same constant. Thus, since we have $\psi\leq\widetilde\psi+\frac{\varepsilon}{4}$ and 
$\psi\geq\widetilde\psi-\frac{\varepsilon}{4}$ in $\Omega_T$, an application of Lemma~\ref{comparisonoflimits} yields 
\[
 |u^*-\widetilde u^*|\leq\frac{\varepsilon}{4}
\]
in $\Omega_T$. 

Next we prove that the function $\widetilde u^*$ is continuous in $Q_{r/4}$. We start by showing that the function 
$\widetilde\varphi_k$ is a weak subsolution in $Q_r$ for every $k\in\N_0$. This is clearly true for $k=0$, since 
$\widetilde\varphi_0=\widetilde\psi$ is a weak solution in $Q_r$. Assume then that the claim holds for some $k\in\N_0$. Now, 
the function $\max\{\widetilde\varphi_k,\widetilde v_k\}$ is a weak subsolution in $Q_r\cap Q^k$ by Lemma~\ref{minsuper}. 
Since trivially $\max\{\widetilde\varphi_k,\widetilde v_k\}\geq\widetilde\varphi_k$, we deduce that 
\[
 \widetilde\varphi_{k+1}=\left\{ \begin{array}{ll}
\max\{\widetilde\varphi_k,\widetilde v_k\} & \textrm{in}\,\,Q_r\cap Q^k\\[2mm]
\widetilde\varphi_k & \textrm{in}\,\,Q_r\setminus Q^k
\end{array}\right.
\]
is a weak subsolution in $Q_r$ by Lemma~\ref{superglued}. 

Extract from the sequence $(\widetilde\varphi_k)_{k=0}^{\infty}$ a subsequence $(\widetilde\varphi_{k_i})_{i=0}^{\infty}$ such that 
$k_0\geq 1$ and 
\[
 \partial_p\left(Q^{k_i-1}\cap\left\{t<t_0+\frac12r^2\right\}\right)\subset Q_r\setminus Q_{r/2}
\]
for every $i\in\N_0$. This can be done, since the collection $\mathcal{F}$ is dense. Next, notice that $\widetilde\varphi_{k_i-1}$ 
is a weak subsolution in $Q^{k_i-1}\cap\left\{t<t_0+\frac12r^2\right\}$ for every $i\in\N_0$, and thus, the comparison principle, 
Lemma~\ref{comparisonprinciple}, yields $\widetilde v_{k_i-1}\geq\widetilde\varphi_{k_i-1}$ in 
$Q^{k_i-1}\cap\left\{t<t_0+\frac12r^2\right\}$. Therefore, $\widetilde\varphi_{k_i}=\widetilde v_{k_i-1}$ in 
$Q^{k_i-1}\cap\left\{t<t_0+\frac12r^2\right\}$, and since $Q_{r/2}\subset Q^{k_i-1}\cap\left\{t<t_0+\frac12r^2\right\}$ for 
every $i\in\N_0$, we see that $\widetilde\varphi_{k_i}$ is a weak solution in $Q_{r/2}$ for every $i\in\N_0$. 

Since $\widetilde\varphi_{k_i}$ is continuous in $\Omega_T$, we may apply Theorem~\ref{aprioritheorem} to obtain 
 \[
 ||D\widetilde\varphi_{k_i}||_{L^\infty(Q_{r/8})}\leq c\left(\mint_{Q_{r/4}}\big[G(|D\widetilde\varphi_{k_i}|)+1\big]\dx\dt\right)^{\max\left\{\frac12,\frac2{(n+2)g_0-2n}\right\}}.
 \]
Moreover, since $|\widetilde\varphi_{k_i}|\leq\sup_{\Omega_T}|\psi|=:M$ in $\Omega_T$ by the maximum principle and Proposition~\ref{properties}, the Caccioppoli inequality, Lemma~\ref{caccioppoli}, yields 
\[
\mint_{Q_{r/4}}G(|D\widetilde\varphi_{k_i}|)\dx\dt\leq c\mint_{Q_{r/2}}\left[G\left(\frac{|\widetilde\varphi_{k_i}|}{r}\right)+\frac{|\widetilde\varphi_{k_i}|^2}{r^2}\right]\dx\dt\leq c(g_0,g_1,\nu,L,M,r)
\]
for every $i\in\N_0$. Thus, we have a uniform bound for $||D\widetilde\varphi_{k_i}||_{L^\infty(Q_{r/8})}$ and since $\widetilde\varphi_{k_i}$ converges pointwise to $\widetilde u^*$, by Lemma~\ref{gradientconvergence} we obtain 
\[
||D\widetilde u^*||_{L^\infty(Q_{r/8})}\leq c(n,g_0,g_1,\nu,L,M,r). 
\]
By applying Theorem~\ref{superlimit} to both $(\widetilde\varphi_{k_i})_{i=0}^{\infty}$ and $(-\widetilde\varphi_{k_i})_{i=0}^{\infty}$ we see that $\widetilde u^*$ is also a weak solution in $Q_{r/2}$. Now the continuity of $\widetilde u^*$ in a neighborhood of $(x_0,t_0)$ follows using a completely analogous proof to Proposition~4.2 in \cite{BL} together with a simple approximation argument.

Finally, it is easy to see that 
\[
 \sup u^* - \sup \widetilde u^* \leq \sup|u^*-\widetilde u^*|
\]
and thus we have 
\[
 \osc u^* - \osc \widetilde u^* = \sup u^* - \sup \widetilde u^* + \sup (-u^*) - \sup (-\widetilde u^*)\leq 2\sup|u^*-\widetilde u^*|.
\]
Since $\widetilde u^*$ is continuous in a neighborhood of $(x_0,t_0)$, we find $0<\delta<r/8$ such that 
\[
 \osc_{Q_{\delta}}\widetilde u^* < \frac{\varepsilon}{2}.
\]
Therefore, by writing 
\[
 \osc_{Q_{\delta}}u^*\leq\osc_{Q_{\delta}} \widetilde u^*+2\sup_{Q_{\delta}}|u^*-\widetilde u^*|<\varepsilon,
\]
we see that $u^*$ is continuous at $(x_0,t_0)$, as required. 
\end{proof}

Let us then conclude by showing that our candidate $u^*$ is a weak supersolution. We shall do this by constructing a sequence of weak supersolutions that converge to $u^*$ almost everywhere and then using Theorem~\ref{superlimit}. Observe that $u^*$ being the limit of Construction~\ref{construction} plays no special role in the proof, in fact, the result holds for any continuous $\A$-superparabolic function. 

Let $K_0\Subset\Omega_T$ be a cube, and let $\{K_k^j\}_{j=1}^{2^{nk}}$, denote the $k^{\textrm{th}}$ generation of 
dyadic subcubes of $K_0$. Set $Q_0:=K_0\times(0,T)$ and $Q_k^j:=K_k^j\times(0,T)$, and define for each $k\in\N$ the function 
$u_k: Q_0\to \R$ such that $u_k$ is a weak solution in $Q_k^j$ and $u_k=u^*$ on $\partial_pQ_k^j$ for every $j\in\{1,\ldots,2^{nk}\}$. 

\begin{Lemma}\label{supersolution1}
 The function $u_k$ is a weak supersolution in $Q_0$ for every $k\in\N$.
\end{Lemma}
\begin{proof}
Let us first show that $u_k$ is $\A$-superparabolic in $Q_0$. Clearly $u_k$ is continuous in $\overline Q_0^p$, since $u^*$ is continuous by 
Lemma~\ref{continuity} and $u_k$ is continuous up to the boundary in each $Q_k^j, j\in\{1,\ldots,2^{nk}\}$, by 
Theorem~\ref{existenceofsolution}. Moreover, as a continuous function $u_k$ is bounded in the compact set $\overline Q_0^p$. Thus, we only 
need to check that $u_k$ satisfies the comparison principle in each cylinder.

Fix a cylinder $Q\Subset Q_0$ and let $h\in C^0(\overline Q^p)$ be a weak solution in $Q$ such that $h\leq u_k$ on $\partial_pQ$. Since $u^*$ is $\A$-superparabolic by Lemma~\ref{A-super} and $u_k$ is a weak solution in $Q_k^j$ with $u_k=u^*$ on $\partial_pQ_k^j$ for every $j\in\{1,\ldots,2^{nk}\}$, we have $u_k\leq u^*$ in $\overline Q_0^p$. Now $h\leq u_k\leq u^*$ on $\partial_pQ$, and therefore another application of the comparison principle yields $h\leq u^*$ in $Q$. This, together with the fact that $u_k=u^*$ on $\partial_pQ_k^j$ for a fixed $j\in\{1,\ldots,2^{nk}\}$, implies that $h\leq u_k$ on $Q\cap\partial_pQ_k^j$. Thus, $h\leq u_k$ also on $\partial_p(Q\cap Q_k^j)$, and since $h$ and $u_k$ are both weak solutions in $Q\cap Q_k^j$, we obtain $h\leq u_k$ in $Q\cap Q_k^j$ by Lemma~\ref{comparisonprinciple}. Since this holds for every $j\in\{1,\ldots,2^{nk}\}$, we have $h\leq u_k$ in the whole $Q$.

Next we prove that $u_k$ is a weak supersolution in $Q_0$. Let $\widetilde Q\subset Q_0$ be a cylinder, and fix 
$r\in\R$ and $i\in\{1,\ldots,n\}$. We show that if $u_k$ is a weak supersolution in both $Q_1:=\{(x,t)\in \widetilde Q: x_i<r\}$ and 
$Q_2:=\{(x,t)\in \widetilde Q: x_i>r\}$, then it is a weak supersolution in $\widetilde Q$. Using this repeatedly will then give the desired result. 

Define for $m\in\N$ the set $U_m:=\left\{(x,t)\in \widetilde Q: r-\frac{1}{m}<x_i<r+\frac{1}{m}\right\}$ and the function 
\[
 v_m = \left\{ \begin{array}{ll}
h_m & \textrm{in}\,\, U_m\\
u_k & \textrm{in}\,\, \widetilde Q\setminus U_m\\
\end{array}\right.,
\]
where $h_m$ is a weak solution in $U_m$ with $h_m=u_k$ on $\partial_pU_m$. Since $u_k$ is $\A$-superparabolic, the 
comparison principle yields $h_m\leq u_k$ in $U_m$. This, together with the fact that $u_k\in C^0(\overline Q_1^p)$ and 
$h_m\in C^0(\overline U_m^p)$ are weak supersolutions in $Q_1$ and $U_m$, respectively, allows us to apply 
Lemma~\ref{superglued} to deduce that $v_m$ is a weak supersolution in $Q_1$. Similar reasoning shows that $v_m$ is a 
weak supersolution also in $Q_2$. Since $v_m$ is a weak (super)solution in $U_m$ and being a 
weak supersolution is a local property, $v_m$ is a weak supersolution in the whole $\widetilde Q$. 

Let $M:=\sup_{\widetilde Q}|u_k|$. By the maximum principle, Corollary~\ref{maximumprinciple}, 
\[
 |h_m|\leq\sup_{\partial_pU_m}|h_m|=\sup_{\partial_pU_m}|u_k|\leq M
\]
in $U_m$, and thus $|v_m|\leq M$ in $\widetilde Q$. Moreover, by the continuity of $u_k$, $v_m$ clearly converges to $u_k$ pointwise in $\widetilde Q$ as $m\to\infty$. Therefore, by Theorem~\ref{superlimit} $u_k$ is a weak supersolution in $\widetilde Q$. This concludes the proof. 
\end{proof}

\begin{Lemma}\label{supersolution2}
 The limit $u^*$ of Construction~\ref{construction} is a weak supersolution in $\Omega_T$.
\end{Lemma}
\begin{proof}
Let $(x,t)$ be any point in $\Omega_T$. Since $\Omega_T$ is open, we can always find a cube $K_0\Subset\Omega$ and $0<t_1<t_2<T$ such that 
$(x,t)\in \frac12 K_0\times(t_1,t_2)=:Q$, where $\frac12 K_0$ denotes the cube with the same center and half the side length as $K_0$. We may assume $\diam(K_0)\leq 1$ without loss of generality. By Lemma~\ref{supersolution1} $u_k$ is a weak supersolution in $Q_0=K_0\times(0,T)$ for every $k\in\N$, and since 
\[
 |u_k|\leq\sup_{Q^j_k}|u_k|\leq\sup_{\partial_pQ^j_k}|u^*|\leq\sup_{\Omega_T}|\psi|=:M
\]
in $Q^j_k$ for every $j\in\{1,\ldots,2^{nk}\}$ by the maximum principle, the sequence $(u_k)_{k=1}^{\infty}$ is 
uniformly bounded in $Q_0$. Therefore, if we manage to show that $(u_k)_{k=1}^{\infty}$ has a subsequence that converges to $u^*$ almost everywhere in $Q$, the result follows by Theorem~\ref{superlimit}. 

To this end, let $\varepsilon>0$. Since $u^*$ is continuous by Lemma~\ref{continuity}, there exists $\eta\in C^{\infty}(\Omega_T)$ such 
that $|\eta-u^*|<\varepsilon$ in $\overline Q^p$. By requiring $k\geq 2$ we may label the cubes $K_k^j$ such that 
\[
\bigcup_{j=1}^{2^{n(k-1)}}K_k^j=\frac12 K_0. 
\]
Set 
\[
 w_k := (\eta-u_k-\varepsilon)_+
\]
and observe that $w_k(\cdot,t)\in W^{1,G}_0(K_k^j)$ for almost every $t\in(t_1,t_2)$ and every $j\in\{1,\ldots,2^{n(k-1)}\}$. Thus, by applying the Poincar\'e's inequality, Lemma~\ref{poincare}, to the function $\diam(K_k^j)w_k$, we obtain 
\[
\begin{split}
 \int_{t_1}^{t_2}\int_{K_k^j}G(w_k)\dx\dt&\leq \int_{t_1}^{t_2}\int_{K_k^j}G\big(\diam(K_k^j)|Dw_k|\big)\dx\dt\\
 &\leq 2^{-g_0k}\int_{t_1}^{t_2}\int_{K_k^j}G(|Dw_k|)\dx\dt,
\end{split}
\]
since we assumed $\diam(K_0)\leq 1$. Therefore, 
\begin{equation}\label{B}
 \begin{split}
  \int_{Q}G(w_k)\dx\dt&=\sum_{j=1}^{2^{n(k-1)}}\int_{t_1}^{t_2}\int_{K_k^j}G(w_k)\dx\dt\\
  &\leq 2^{-g_0k}\sum_{j=1}^{2^{n(k-1)}}\int_{t_1}^{t_2}\int_{K_k^j}G(|Dw_k|)\dx\dt\\
  &=2^{-g_0k}\int_{Q}G(|Dw_k|)\dx\dt.
 \end{split}
\end{equation}

Let $\varphi\in C^{\infty}_c(Q_0)$ be such that $0\leq\varphi\leq 1$, $\varphi=1$ in $Q$, and 
\[
\left|\left|\partial_t\varphi\right|\right|_{L^{\infty}(Q_0)},||D\varphi||_{L^{\infty}(Q_0)}\leq C
\]
for some $C\geq 1$. By applying the Caccioppoli estimate, Lemma~\ref{caccioppoli}, to the nonnegative weak subsolution 
$M-u_k$ we conclude, similarly as in the proof of Lemma~\ref{gradientconvergence}, that 
\[
 \int_Q G(|Du_k|)\dx\dt \leq c\left(M^2C+M^{g_1}C^{g_1}\right)|Q_0|.
\]
Thus, combining this with \eqref{B} yields 
\[
 \begin{split}
  |Q\cap\{w_k>\delta\}|&=\int_{Q\cap\{w_k>\delta\}}G(\delta^{-1}\delta)\dx\dt\\
  &\leq\delta^{-g_1}\int_Q G(w_k)\dx\dt\leq 2^{-g_0k}\delta^{-g_1}\int_{Q}G(|Dw_k|)\dx\dt\\
  &\leq 2^{-g_0k}\left(\frac{2}{\delta}\right)^{g_1}\left(\int_{Q}G(|D\eta|)\dx\dt+\int_{Q}G(|Du_k|)\dx\dt\right)\\
  &\leq 2^{-g_0k}c\left(\delta,g_0,g_1,\nu,L,M,C,|Q_0|,\|D\eta\|_{L^{\infty}(Q_0)}\right)
 \end{split}
\]
for every $\delta\in(0,1)$, implying that $w_k\to 0$ in measure in $Q$ as $k\to\infty$. Therefore, we find a subsequence 
$(w_{k_i})_{i=1}^{\infty}$ converging to zero almost everywhere in $Q$ as $i\to\infty$. Now, together with the choice of $\eta$, this 
gives 
\[
 \lim_{i\to\infty}u_{k_i}\geq\eta-\varepsilon>u^*-2\varepsilon
\]
almost everywhere in $Q$. Since $u_k\leq u^*$ in $Q_0$ for every $k\in\N$, we obtain $\lim_{i\to\infty}u_{k_i}=u^*$ almost everywhere in $Q$ by letting $\varepsilon\to 0$. This finishes the proof. \qedhere

\end{proof}

\begin{Theorem}\label{obstaclecontinuity}
 Let $\psi$ be a continuous function in $\overline\Omega_T^p$. Then the solution to the obstacle problem with the obstacle~$\psi$ in $\Omega_T$ is continuous.
\end{Theorem}
\begin{proof}
By Lemma \ref{supersolution2} the limit $u^*$ of Construction~\ref{construction} is a weak supersolution in $\Omega_T$. Moreover, by Proposition~\ref{properties} $u^*\geq\psi$ in $\Omega_T$, and if also $v$ is an $\essliminf$-regularized weak supersolution with $v\geq\psi$ almost everywhere in $\Omega_T$, then $v\geq u^*$ in $\Omega_T$. Thus, $u^*$ must coincide with the unique solution to the obstacle problem given by Theorem~\ref{bddexistence}. Continuity follows from Lemma~\ref{continuity}. 
\end{proof}

We end the section by showing that whenever the solution to the obstacle problem lies strictly above the obstacle, it is in fact a weak solution. We shall prove this for the limit $u^*$ of Construction~\ref{construction}, which is a solution as seen above. First we need the following lemma. 

\begin{Lemma}\label{subsolution}
 The function $\varphi_k$ is a weak subsolution in the set $\{\varphi_k>\psi\}$ for every $k\in\N$.
\end{Lemma}
\begin{proof}
If $\varphi_k=\psi$ for some $k$, the set $\{\varphi_k>\psi\}$ is empty and there is nothing to prove. Suppose thus $\varphi_k\neq\psi$ for every $k\in\N$. In the set $\{\varphi_1>\psi\}$ we must have $\varphi_1=v_0$ and hence the claim holds for $k=1$. Assume then that $\varphi_k$ is a weak subsolution in the set $\{\varphi_k>\psi\}$ for some $k\in\N$. First take a point $z_0=(x_0,t_0)\in\{\varphi_k>\psi\}$. The set 
$\{\varphi_k>\psi\}$ is open by the continuity of $\varphi_k$ and $\psi$ and thus, we find a cylinder 
$Q\Subset\{\varphi_k>\psi\}\subset\{\varphi_{k+1}>\psi\}$ such that $z_0\in Q$. Now, the function $\max\{\varphi_k,v_k\}$ is a weak 
subsolution in $Q\cap Q^k$ by Lemma~\ref{minsuper}, and thus, 
\[
 \varphi_{k+1}=\left\{ \begin{array}{ll}
\max\{\varphi_k,v_k\} & \textrm{in}\,\,Q\cap Q^k\\[2mm]
\varphi_k & \textrm{in}\,\,Q\setminus Q^k
\end{array}\right.
\]
is a weak subsolution in $Q$ by Lemma~\ref{superglued}. 

Let then $z_0\in\{\varphi_{k+1}>\psi\}\setminus\{\varphi_k>\psi\}$. Since we always have 
$\varphi_k\geq\psi$ and $z_0\notin\{\varphi_k>\psi\}$, we deduce that $\varphi_k(z_0)=\psi(z_0)$. Thus, $\varphi_{k+1}(z_0)=v_k(z_0)$ or otherwise 
$\varphi_{k+1}(z_0)>\psi(z_0)$ would fail to hold. Next, denote $w_k:=v_k-\varphi_k$ and choose a small enough $\varepsilon>0$ such that 
$w_k(z_0)>\varepsilon$. This can be done, since 
\[
 w_k(z_0)=v_k(z_0)-\varphi_k(z_0)=\varphi_{k+1}(z_0)-\psi(z_0)>0.
\]
By the continuity of $v_k$ and $\varphi_k$ there exists $r>0$ such that $|w_k(z)-w_k(z_0)|<\varepsilon$ for every 
$z\in Q_r:=K_r(x_0)\times(t_0-r,t_0+r)$. Since $\{\varphi_{k+1}>\psi\}$ is open, we may also assume that 
$Q_r\Subset\{\varphi_{k+1}>\psi\}$. Combining these facts yields $w_k(z)>w_k(z_0)-\varepsilon>0$ for every $z\in Q_r$, showing that $v_k>\varphi_k$ in $Q_r$. Therefore, $\varphi_{k+1}=v_k$ is a weak solution in $Q_r$. 

We have now found a neighborhood in which $\varphi_{k+1}$ is a weak subsolution for every point of the set $\{\varphi_{k+1}>\psi\}$. The 
induction argument finishes the proof. 
\end{proof}

\begin{Proposition}
 The limit $u^*$ of Construction~\ref{construction} is a weak solution in the set $\{u^*>\psi\}$.
\end{Proposition}
\begin{proof}
If $u^*=\psi$ there is nothing to prove. Thus, suppose $u^*\neq\psi$ and let $z_0=(x_0,t_0)\in\Omega_T$ be such that $u^*(z_0)>\psi(z_0)$. By the continuity of $u^*$ and $\psi$ the set $\{u^*>\psi\}$ is open, and thus, we find a cylinder $Q\Subset\{u^*>\psi\}$ such that $z_0\in Q$. Since the sequence $(\varphi_k)_{k=0}^{\infty}$ is nondecreasing, we clearly have for every $k\in\N_0$ that 
\[
 \{\varphi_k>\psi\}\subset\{\varphi_{k+1}>\psi\}\subset\{u^*>\psi\}.
\]
Moreover, if $z\notin\bigcup_{k=0}^\infty\{\varphi_k>\psi\}$, we must have $\varphi_k(z)=\psi(z)$ for every 
$k\in\N_0$, and thus $z\notin\{u^*>\psi\}$. Therefore, we see that 
\[
 \{u^*>\psi\} = \bigcup_{k=0}^\infty\{\varphi_k>\psi\}.
\]
In particular, the sets $\{\varphi_k>\psi\}$ form an open cover for $\overline Q^p$, and since $\overline Q^p$ is compact, there exists an integer $k_0\geq 1$ such that $Q\subset\{\varphi_{k_0}>\psi\}$. Now Lemma~\ref{subsolution} implies that $\varphi_k$ is a weak subsolution in $Q$ for every $k\geq k_0$, and since the sequence $(\varphi_k)_{k=0}^{\infty}$ is uniformly bounded, we obtain by applying 
Lemma~\ref{superlimit} to the sequence $(-\varphi_k)_{k=k_0}^{\infty}$ that $u^*$ is a weak subsolution in $Q$. Therefore, $u^*$ is a weak 
subsolution in $\{u^*>\psi\}$, and since $u^*$ is a weak supersolution in $\Omega_T$ by Lemma~\ref{supersolution2}, the result follows. 
\end{proof}

\vs

{\bf Acknowledgements}. This research has been supported by the Vilho, Yrj\"o and Kalle V\"ais\"al\"a Foundation.

\end{document}